\documentclass[10pt]{article}
\usepackage[a4paper, total={5.5in, 8in}]{geometry}
\usepackage{amsthm}
\usepackage{mathrsfs}
\usepackage{comment}

\usepackage{amsmath,amsfonts,amssymb,color,enumitem,tensor,mathtools,graphicx,subcaption,xcolor}
\usepackage[only,llbracket,rrbracket,llparenthesis,rrparenthesis]{stmaryrd} 
\usepackage[textsize=tiny,color=blue!10!white]{todonotes}
\usepackage{bm}
\usepackage[medium,compact]{titlesec}
\titleformat*{\section}{\large\bfseries}
\titleformat*{\subsection}{\bfseries}
\usepackage[breaklinks,bookmarks=false]{hyperref}
\hypersetup{colorlinks, linkcolor=blue, citecolor=blue,
urlcolor=blue, plainpages=false, pdfwindowui=false,
pdfstartview={FitH}}

\numberwithin{equation}{section}

\newtheorem{theorem}{Theorem}[section]{\bfseries}{\it}
{\bfseries}{\it}
\newtheorem{lemma}[theorem]{Lemma}{\bfseries}{\it}
\newtheorem{corollary}[theorem]{Corollary}{\bfseries}{\it}
{\bfseries}{\it}
{\bfseries}{\rmfamily}
{\bfseries}{\it}
\theoremstyle{definition}
\newtheorem{remark}{Remark}[section]{\bfseries}{\rmfamily}
\newcommand{\proofbox}{\qed}

\newcommand{\R}{\mathbb{R}}
\newcommand{\N}{\mathbb{N}}

\newcommand{\calVk}{\mathcal{V}_k}
\newcommand{\calVki}{\mathcal{V}_{k,\Omega}}
\newcommand{\Tk}{\mathcal{T}_k}
\DeclareMathOperator{\Card}{card}

\DeclareMathOperator{\diam}{diam}
\DeclareMathOperator{\Div}{div}
\renewcommand{\dim}{d}

\newcommand{\Dk}{{D_k}}

\newcommand{\abs}[1]{\lvert#1\rvert}
\newcommand{\norm}[1]{\lVert#1\rVert}
\newcommand{\Mfty}{M_\infty}

\newcommand{\Rddsym}{\R^{\dim\times\dim}_{\mathrm{sym}}}
\newcommand{\Rddsymp}{\R^{\dim\times\dim}_{\mathrm{sym},+}}
\newcommand{\CD}{C_{D}}
\newcommand{\Cstab}{C_{\mathrm{stab}}}

\newcommand{\lamk}{\lambda_k}
\newcommand{\mkp}{m_{k,\lamk}^*}
\newcommand{\Vkp}{V_{k,+}}

\newcommand{\klk}{k,\lambda_k}

\newcommand{\RkH}{{R_{k,\lambda_k}^1}}
\newcommand{\RkF}{{R_{k,\lambda_k}^2}}
\newcommand{\Creg}{C_{\Omega,\gamma}}
\newcommand{\Rk}{{\mathcal{R}_k}}

\title{Rates of convergence of finite element approximations of second-order mean field games with nondifferentiable Hamiltonians}
  \author{Yohance A. P. Osborne\footnotemark[1]~ and Iain Smears\footnotemark[2]}

\begin{document}

\maketitle

\renewcommand{\thefootnote}{\fnsymbol{footnote}}

\footnotetext[1]{Department of Mathematical Sciences, Durham University, Stockton Road, DH1 3LE Durham, United Kingdom (\texttt{yohance.a.osborne@durham.ac.uk}).}
\footnotetext[2]{Department of Mathematics, University College London, Gower
	Street, WC1E 6BT London, United Kingdom (\texttt{i.smears@ucl.ac.uk}).}

\begin{abstract}
{We prove a rate of convergence for finite element approximations of stationary, second-order mean field games with nondifferentiable Hamiltonians posed in general bounded polytopal Lipschitz domains with strongly monotone running costs. In particular, we obtain a rate of convergence in the $H^1$-norm for the value function approximations and in the $L^2$-norm for the approximations of the density. We also establish a rate of convergence for the error between the exact solution of the MFG system with a nondifferentiable Hamiltonian and the finite element discretizations of the corresponding MFG system with a regularized Hamiltonian.}
\end{abstract}

\section{Introduction}\label{sec:introduction}

Mean field games, which were introduced by Lasry and Lions~\cite{lasry2006jeuxI,lasry2006jeuxII,lasry2007mean}, and independently by Huang, Caines and Malham\'e~\cite{huang2006large}, are models of Nash equilibria of differential games of stochastic optimal control in which there are infinitely many players.
{Under appropriate assumptions, the Nash equilibria of the game} are characterized by a system of {nonlinear} PDE in which a Hamilton--Jacobi--Bellman (HJB) equation for the value function of the optimal control problem faced by a typical player is coupled with a Kolmogorov--Fokker--Planck (KFP) equation that determines the player density over the state space of the game.
For an introduction to the theory and applications of mean field games, we refer the reader to the extensive reviews~\cite{lasry2007mean,GueantLasryLions2003,GomesSaude2014,achdou2020mean}.
{In the literature on these problems, it is very often assumed that the Hamiltonian of the HJB equation is differentiable with respect to the variable for the gradient of the value function, since} the derivative of the Hamiltonian {appears in} the advective term of the KFP equation.
{This assumption is not merely a technical one, since the differentiability of the Hamiltonian is closely connected to the properties of the optimal feedback controls of the underlying control problem, with nondifferentiability of the Hamiltonian implying nonuniqueness of the optimal feedback controls.}
{However, in many applications of interest the optimal feedback controls are not necessarily unique and the Hamiltonian may fail to be differentiable, which is common for instance in minimal exit-time problems and problems with \emph{bang-bang} controls}.
Recently in~\cite{osborne2024thesis,osborne2022analysis,osborne2024erratum,osborne2023finite,osborne2024regularization}, we introduced a generalization of the MFG system of Lasry and Lions to accommodate general nondifferentiable Hamiltonians.
The KFP \emph{equation} is generalized to a partial differential \emph{inclusion} (PDI) involving the subdifferential of the Hamiltonian{, which expresses the idea that the players' choices among possibly nonunique optimal controls may no longer be given \emph{a priori}, but become an implicit part of the problem to be solved to find a Nash equilibrium.}

As a model problem, we consider here the MFG PDI from~\cite{osborne2022analysis}, which can be written formally as
\begin{equation}\label{eq:mfg_pdi_sys}
\begin{aligned}
- \nu\Delta u+H(x,\nabla u) &=F[m](x)  && \text{in }\Omega,
\\
-\nu\Delta m -  G(x) &\in \Div(m \partial_p H(x,\nabla u)) &&\text{in }\Omega,
\end{aligned}
\end{equation}
where the domain~$\Omega \subset \R^d$ denotes the state-space of the game in which the players' states evolve, where $u\colon \Omega \to {}\R$ denotes the value function and where $m\colon \Omega\to \R$ denotes the density of players.
We will consider~\eqref{eq:mfg_pdi_sys} along with homogeneous Dirichlet boundary conditions on $u$ and $m$ for simplicity, which arise in models where the players exit the game upon reaching the boundary $\partial \Omega$.
The constant $\nu>0$ relates to the strength of the stochastic noise in the player's dynamics.
The Hamiltonian $H\colon (x,p)\mapsto H(x,p)$ is a real-valued function on $\Omega\times \R^d$ that is convex {and Lipschitz} in its second argument and is obtained from the underlying optimal control problem of the players. 
The subdifferential $\partial_p H$ of the Hamiltonian maps points in $\Omega\times \R^d$ to nonempty closed convex subsets of $\R^d $, and, under modest assumptions, can be shown to be equal to the convex hull of the set of drifts formed by the optimal feedback controls, see for instance~\cite[Lemma~3.4.1]{osborne2024thesis}.
The coupling term $F$ represents the component of the players' cost functional that depends on the overall density, and is allowed to be a possibly nonlinear and nonlocal operator on the density function.
The source term $G$ in the KFP equation arises in situations where new players are entering the game at a certain rate per unit time.
Precise assumptions on the data will be stated in Section~\ref{sec:model_data} below. 
In practice, the problem~\eqref{eq:mfg_pdi_sys} is understood in a suitable weak sense with the inclusion formulated in terms of measurable selections from the subdifferential set.
The existence and uniqueness of weak solutions has been shown in~\cite{osborne2022analysis,osborne2024erratum,osborne2023finite} under suitable assumptions on the problem data, with the key results summarized in Section~\ref{sec-eq:PDI_weakform-def-and-disc} below.
We also refer the reader to~\cite{osborne2024regularization} for a derivation of MFG PDI in terms of the underlying optimal problem, and for an analysis showing how these characterize the limit when considering sequences of problems with differentiable Hamiltonians that are converging to a nondifferentiable limit.
We also mention the earlier work~\cite{ducasse2020second} which treated the case of a density-dependent Hamiltonian of the form $H(x,p,m)=K(x,m)|p|$ for some function $K$.

In~\cite{osborne2022analysis,osborne2023finite}, we introduced a class of piecewise linear stabilized finite element methods (FEM) for these problems, in both the steady-state and time-dependent settings.
In particular, under suitable assumptions, we proved the convergence, without rates, of the numerical approximations in the small-mesh limit (and small time-step size for the time-dependent case). 
More precisely, for the steady-state problem, we proved convergence in the $H^1$-norm for the value function, and in $L^q$-norms for the density, for a range of values of $q$ determined by the Sobolev embedding theorem.
{We mention that \emph{qualitative} results of a similar nature}, i.e.\ plain convergence in some suitable norms, {have been obtained previously for problems with differentiable Hamiltonians in earlier works} for a variety of discretization methods, including finite difference and semi-Lagrangian methods~\cite{achdou2013mean,achdou2010mean,achdou2016convergence,CarliniSilva14,CarliniSilve15}.
So far \emph{quantitative} convergence results in the form of error bounds have only been obtained for MFG systems with more regular Hamiltonians in~\cite{bonnans2022error,berry2025approximation,osborne2024near,osborne2025posteriori}.
In~\cite{bonnans2022error}, rates of convergence are shown for finite difference approximations of classical solutions for a time-dependent MFG system.
In~\cite{osborne2024near}, we proved the near asymptotic quasi-optimality and optimal rates of convergence of the stabilized finite element methods mentioned above, for problems with regular Hamiltonians and strongly monotone couplings.
We also mention that~\cite{berry2025approximation} considered error bounds for finite element approximations of a class of MFG that is based on a different approach via linearizations around classical stable solutions of the problem.
In~\cite{osborne2025posteriori}, we obtained the first \emph{a posteriori} error bounds for these problems, still in the context of regular Hamiltonians. 

In this paper, we prove \emph{a priori} error bounds for the stabilized FEM discretizations of~\eqref{eq:mfg_pdi_sys} that were introduced previously in~\cite{osborne2022analysis}.
Under suitable assumptions detailed below, for problems on a bounded polytopal domain $\Omega \subset \R^d$ with Lipschitz boundary, meshed by a shape-regular sequence of simplicial meshes $\{\mathcal{T}_k\}_{k\in \N}$ on $\Omega$, we show, in~Theorem~\ref{thm:main_convergence_rate} below, a rate of convergence 
\begin{equation}\label{eq:apriori_intro}
\norm{u-u_k}_{H^1(\Omega)}+\norm{m-m_k}_{L^2(\Omega)} \lesssim h_k^{\gamma/3},
\end{equation}
for all $k$ sufficiently large, where $(u_k,m_k)$ denotes the corresponding finite element approximation of~$(u,m)$ as defined in~\cite{osborne2022analysis}, and where the exponent $\gamma\in (0,1]$ depends on the domain~$\Omega$ via the elliptic regularity for Poisson's equation on $\Omega$ with $L^2$-regular source terms, where we only assume here that $\Omega$ is bounded, polytopal, and Lipschitz. Note that $\gamma=1$ in the case of convex domains, in which case we obtain a rate of order $1/3$ from~\eqref{eq:apriori_intro}.
The proof of~\eqref{eq:apriori_intro} requires addressing multiple challenges.
On top of the nonlinear coupling between the equations and their lack of coercivity, a further fundamental difficulty is that the term in~\eqref{eq:mfg_pdi_sys} involving the subdifferential of the Hamiltonian is quasilinear with respect to the gradient variable, is only determined implicitly via the inclusion, and is generally discontinuous with respect to its arguments as a result of the possible lack of differentiability of the Hamiltonian. 
{Thus, in general, there is no quantitative control between general measurable selections from the subdifferentials $\partial_p H(x,\nabla u_k)$ and those from~$\partial_p H(x,\nabla u)$.} 
There is also the problem that, in general, measurable selections from the subdifferential $\partial_p H(x,\nabla u)$ in~\eqref{eq:mfg_pdi_sys} have low regularity, so the density function $m$ may also have very low regularity, even in one space-dimension, see~\cite[Example~2]{osborne2022analysis} for a concrete example.
Addressing these challenges therefore requires significant new ideas beyond the techniques that we developed in~\cite{osborne2024near} for the quasi-optimality results for problems with more regular Hamiltonians. 
{We also note that the convergence rate in~\eqref{eq:apriori_intro} treats the $L^2$-norm of the error in the density approximation rather than the $H^1$-norm, which is natural given that the convergence in norm of the gradients of density approximations is still an open problem as a result of the lack of continuity in the subdifferential sets of $H$, see~\cite[p.~151]{osborne2022analysis} for further discussion.}

{The basis of our approach to the proof of~\eqref{eq:apriori_intro} is based on our analysis of regularizations of the PDI} in~\cite{osborne2024regularization}, where we showed that the solution $(u,m)$ of~\eqref{eq:mfg_pdi_sys} can be approximated by the solution $(u_{\lambda},m_{\lambda})$ of the regularized MFG system of partial differential \emph{equations} 
\begin{equation}\label{eq:regularized_MFG}
	\begin{aligned}
		- \nu\Delta {u}_{\lambda}+{H}_{\lambda}(x,\nabla {u}_{\lambda})&={F}[{m}_{\lambda}](x) &&\text{in }\Omega,
		\\
		-\nu\Delta {m}_{\lambda} - \text{div}\left({m}_{\lambda}\frac{\partial {H}_{\lambda}}{\partial p}(x,\nabla {u}_{\lambda})\right) &= {G}(x) &&\text{in }\Omega, 
	\end{aligned}
\end{equation}
where the regularized Hamiltonians $H_{\lambda}$, $\lambda\in (0,1]$, is a suitably defined $C^{1,1}$ approximation of $H$, such that $H_{\lambda}\to H$ as $\lambda \to 0$.
For problems with strongly monotone couplings, we proved in~\cite{osborne2024regularization} that
\begin{equation}\label{eq:intro_reg_err_bound}
	\| u-u_{\lambda} \|_{H^1(\Omega)} + \|m-m_{\lambda}\|_{L^2(\Omega)}\lesssim  \omega(\lambda)^{\frac{1}{2}}, 
\end{equation}
{where $\omega(\lambda)\coloneqq \sup_{(x,p)\in\overline{\Omega}\times\mathbb{R}^d}|H(x,p)-H_{\lambda}(x,p)|$ quantifies the error introduced by regularization of the Hamiltonian,} {and where the hidden constant in~\eqref{eq:intro_reg_err_bound} is independent of $\lambda$.}
{In this work, we consider regularization of the Hamiltonian by the well-known Moreau--Yosida regularization, as detailed in Section~\ref{sec:regularization} below, in which case $\omega(\lambda) $ is of order $\lambda$, so~\eqref{eq:intro_reg_err_bound} gives an upper bound of order $\lambda^{\frac{1}{2}}$ for the error between $(u,m)$ and $(u_{\lambda},m_{\lambda})$.}
We will show below that a similar bound of order $\lambda^{\frac{1}{2}}$ also holds between the finite element approximation $(u_k,m_k)$ of the MFG PDI system~\eqref{eq:mfg_pdi_sys} and the finite element approximation $(u_{\klk},m_{\klk})$ of the regularized PDE system~\eqref{eq:regularized_MFG} where, for each $k\in \N$, {$\lambda_k\in (0,1]$} is a {mesh-dependent} regularization parameter associated to the mesh~$\Tk$, such that $\lambda_k\to 0$ as $k\to \infty$, see Lemma~\ref{lem:pdi_reg_discrete_lambda_rate} below. 
{Thus the analysis in this work can be summarized by the  diagram
\begin{equation}\label{eq:bounds_diagram}
\begin{array}{c c c}

(u,m) & & (u_k,m_k)
\\[0.2cm] \hspace{1.3cm} \Biggl \updownarrow \text{order } \lambda_k^{\frac{1}{2}}   & & \hspace{1.3cm} \Biggl \updownarrow  \text{order }  \lambda_k^{\frac{1}{2}}
   \\[0.6cm]
(u_{\lambda_k},m_{\lambda_k}) & \xleftrightarrow[ \text{order } \lambda_k^{\frac{1}{2}} + \lambda_k^{-1} h_k^\gamma  ]{}  & (u_{\klk},m_{\klk})
\end{array}
\end{equation}
where the indicated orders refer to the corresponding bounds for the $L^2$-norm of the difference in the density variables, and the bounds hold for $k$ sufficiently large. 
Note that the corresponding bounds for the $H^1$-norm of the value functions can be obtained from the $L^2$-norm bounds of the density functions via the stability of the HJB equation, c.f.\ Lemma~\ref{lem:PDI_HJB_stability} below.}
{It then follows from the triangle inequality, the stability of the HJB equation, and the above bounds that} $\norm{u-u_k}_{H^1(\Omega)}+\norm{m-m_k}_\Omega \lesssim \lambda_k^{\frac{1}{2}}+\lambda_k^{-1}h_k^\gamma$, {which yields~\eqref{eq:apriori_intro} upon choosing $\lambda_k= h_k^{2\gamma /3}$, c.f.\ Theorem~\ref{thm:main_convergence_rate} and its proof below.}
{The main challenge addressed here is to obtain the bound between $(u_{\lamk},m_{\lamk})$ and its approximation $(u_{\klk},m_{\klk})$, i.e.\ the lower part of the diagram~\eqref{eq:bounds_diagram} above, c.f.\ Lemma~\ref{lem:discrete_reg_l2_bound} below.}
{Note that the appearance of the negative power of $\lambda_k$ in this bound is to be expected} due to the fact that the Lipschitz constant of the partial derivative $\frac{\partial H_{\lambda_k}}{\partial p}$ grows as $\lambda_k^{-1}$ and thus diverges as $k\to \infty$.
{As a further contribution of this work, we also obtain rates of convergence between $(u,m)$ and $(u_{\klk},m_{\klk})$, see Corollary~\ref{cor:exact_to_discrete_reg} below, which is also particularly relevant for practical computations. Indeed, in practice it can be easier to compute $(u_{\klk},m_{\klk})$ instead of $(u_k,m_k)$ since the lack of continuity of the discretized inclusion problem is a challenge for discrete solvers. Thus, the analysis shows that regularization of the Hamiltonian can be used in practical computations whilst also retaining some quantitative control on the error.}

{
The outline of this paper is as follows. In Section \ref{sec-notation-fem-spaces} we set notation and detail assumptions on the model data for the MFG PDI \eqref{eq:mfg_pdi_sys}, as well as lay the setting for its discretization by a class of stabilized finite element methods. Section \ref{sec-eq:PDI_weakform-def-and-disc} covers the weak formulation of the MFG PDI and its finite element discretization. The weak formulation of the regularized problem \eqref{eq:regularized_MFG} and its finite element discretization are given in Section \ref{sec:regularization}. We dedicate Section \ref{sec:main_results} to presenting our main result Theorem~\ref{thm:main_convergence_rate} and its corollary. Section \ref{sec-proof-of-key-lemma} is devoted to proving the main results, in particular the crucial Lemma~\ref{lem:discrete_reg_l2_bound}.}

\section{Notation and setting}\label{sec-notation-fem-spaces}

\paragraph{{Basic notation.}} We denote $\mathbb{N}\coloneqq \{1,2,3,\cdots\}$.
Let the space dimension $\dim\in \mathbb{N}$ be fixed.
Let $\Rddsym$ denote the space of symmetric matrices in $\R^{\dim\times\dim}$, and we let $\Rddsymp$ denote the cone of positive semi-definite matrices in $\Rddsym$. We equip the space of matrices $\R^{\dim\times\dim}$ with the Frobenius norm $|\cdot|$.
For a Lebesgue measurable set $\omega \subset \mathbb{R}^d$, {let $(\cdot,\cdot)_{\omega}$ denote the standard inner product for scalar functions in $L^2(\omega)$ and vector fields in $L^2(\omega;\R^\dim)$, where the arguments will distinguish between these two cases, and let $\lVert \cdot \rVert_{\omega}$ denote the $L^2$-norm induced by this inner product.} Let $\Omega$ be a open, bounded, connected polytopal subset of $\mathbb{R}^d$ with Lipschitz boundary $\partial \Omega$.
{In the following, we shall use the classical Sobolev spaces $H^s(\Omega)$, $s\geq 0$, with corresponding norms $\norm{\cdot}_{H^s(\Omega)}$; in particular, the fractional-order Sobolev space $H^s(\Omega)$ for noninteger $s$ can be defined by interpolation of spaces, c.f.~\cite{AdamsFournier03}.
Let $H^1_0(\Omega)$ denote the closure of $C^\infty_0(\Omega)$ in $H^1(\Omega)$. The Poincar\'e inequality shows that there exists a nonnegative constant $C_P$ such that $\ \norm{v}_{\Omega}\leq C_P \norm{\nabla v}_\Omega$ for all $v\in H^1_0(\Omega)$. 
Let $H^{-1}(\Omega)$ denote the dual space of $H^1_0(\Omega)$, with natural norm $\norm{\Psi}_{H^{-1}(\Omega)}\coloneqq \sup_{\norm{\nabla v}_\Omega\leq 1} \langle \Psi, v\rangle_{H^{-1}\times H^1_0} $, where $\langle \cdot,\cdot \rangle_{H^{-1}\times H^1_0}$ denotes the duality pairing between $H^{-1}(\Omega)$ and $H^1_0(\Omega)$, and where the supremum ranges over all $v\in H^1_0(\Omega)$.
}

{\paragraph{Regularity of solutions of Poisson's equation.}}
{In the following analysis, we will use some known regularity results for the solution of Poisson's equation in the scale of Sobolev spaces. 
The main property that we will use here is that there exists some $\gamma \in [\frac{1}{2},1]$ and a constant $\Creg$, depending only on $\Omega$
 and on $\gamma$, such that, given any $g\in L^2(\Omega)$,the unique $w\in H^1_0(\Omega)$ that solves $(\nabla w,\nabla v)_{\Omega}=(g,v)_{\Omega}$ for all $v\in H^1_0(\Omega)$ satisfies $w\in H^{1+\gamma}(\Omega)$, and
\begin{equation}\label{eq:elliptic_regularity}
\norm{w}_{H^{1+\gamma}(\Omega)} \leq \Creg \norm{g}_{L^2(\Omega)}.
\end{equation}
The convergence rates that we will derive below will be expressed in terms of the regularity order $\gamma$. 
In particular, the fact that one can take $\gamma\geq \frac{1}{2}$ is a consequence of the $H^{3/2}$-regularity of solutions of Poisson's equation with a homogeneous Dirichlet boundary condition and source term in $L^2(\Omega)$ on general Lipschitz domains, see~\cite[Theorem~B]{JerisonKenig1995}.
In the case of a convex domain $\Omega$, it is well-known that one can take $\gamma=1$ as a result of the $H^2$-regularity of the solution in this case, see for instance~\cite[Theorem~3.2.1.2]{Grisvard2011}.
In the case of a nonconvex polygonal domain $\Omega$ in two space dimensions without slits, solutions have regularity in $H^{s}(\Omega)$ for all $s < 1 + \frac{\pi}{\omega}$, where $\omega$ denotes the maximum interior angle formed at the corners of the boundary of $\partial\Omega$; this can be deduced as a consequence of~\cite[Theorem~4.4.3.7]{Grisvard2011}.
Remark that $\frac{1}{2}<\frac{\pi}{\omega}<1$ when $\Omega$ is a nonconvex polygon. In such a case one can then take any $\gamma < \frac{\pi}{\omega}$.
}

\subsection{Model data}\label{sec:model_data}

Let the diffusion $\nu>0$ be constant, and let $G\in H^{-1}(\Omega)$ be of the form $G=g_0-\nabla \cdot {g_1}$ with $g_0\in L^{q/2}(\Omega)$ and ${g_1}\in L^q(\Omega;\mathbb{R}^d)$ for some $q>d$. We further assume that $G\in H^{-1}(\Omega)$ is nonnegative in the sense of distributions, i.e.\
$\langle G, v\rangle_{H^{-1}\times H_0^1} \geq 0$ for all functions $v\in H_0^1(\Omega)$ that are nonnegative a.e.\ in $\Omega$.
We suppose that $F:L^2(\Omega)\to L^2(\Omega)$ is a Lipschitz continuous operator that satisfies
\begin{equation}\label{eq:F_lipschitz}
	\|F[m_1]-F[m_2]\|_{\Omega}\leq L_F\|m_1-m_2\|_{\Omega}\quad\forall m_1,m_2\in  L^2(\Omega).
\end{equation}
for some constant $L_F\geq 0$. It is clear therefore that
\begin{equation}\label{eq:F_linear_growth}
	\|F[\overline{m}]\|_{\Omega}\leq C_F\left(\|\overline{m}\|_{\Omega}+1\right)\quad\forall \overline{m}\in L^2(\Omega),
\end{equation}
where $C_F\geq 0$ is a constant. In addition, we assume that $F$ is strongly monotone on $H^1_0(\Omega)$ w.r.t {to the norm $\norm{\cdot}_{\Omega}$}, in the sense that there exists a constant $c_{F}>0$ such that 
\begin{equation}\label{eq:F_strong_mono}
	 (F[m_1]-F[m_2],m_1-m_2)_{\Omega} \geq c_{F}\|m_1-m_2\|_{\Omega}^{2},
\end{equation}for all $m_1,\,m_2\in H_0^1(\Omega)$.
{We also stress that this strong monotonicity condition is only required for arguments of $F$ in the smaller space $H^1_0(\Omega)$, and not the whole space $L^2(\Omega)$. We refer the reader to~\cite[Section~2]{osborne2022analysis} for} some concrete examples of the operator $F$ that satisfy these conditions.

{Motivated by the underlying optimal control problem, we suppose that the Hamiltonian~$H$ appearing in the MFG PDI system~\eqref{eq:mfg_pdi_sys} is defined by
	\begin{equation}\label{eq:Hamiltonian_def}
		H(x,p)\coloneqq \sup_{\alpha\in\mathcal{A}}\left\{b(x,\alpha)\cdot p-f(x,\alpha)\right\}\quad \forall (x,p)\in\overline{\Omega}\times\mathbb{R}^d,
	\end{equation}
	where} $b$ and $f$ are uniformly continuous on $\overline{\Omega}\times \mathcal{A}$ with $\mathcal{A}$ a compact metric space.
It follows that the Hamiltonian $H$ is Lipschitz continuous and satisfies 
\begin{equation}\label{eq:Hamiltonian_Lipschitz}
	|H(x,p)-H(x,q)|\leq L_H|p-q|\quad\forall(x,p,q)\in\overline{\Omega}\times\mathbb{R}^d\times\mathbb{R}^d,
\end{equation}
with {some constant $L_H$}, {for instance one can take $L_H\coloneqq\|b\|_{C(\overline{\Omega}\times\mathcal{A};\mathbb{R}^d)}$.} We assume throughout this article that $L_H>0$, since otherwise the analysis that we develop in the sequel will carry through via a substantially simpler argument.
We deduce from~\eqref{eq:Hamiltonian_Lipschitz} that there exists a constant $C_H\geq 0$ such that
\begin{equation}\label{eq:Hamiltonian_linear_growth}
	|H(x,p)|\leq C_H(|p|+1)\quad\forall (x,p)\in \overline{\Omega}\times\mathbb{R}^d.
\end{equation}
It is clear from~\eqref{eq:Hamiltonian_Lipschitz} that the mapping $v\mapsto H(\cdot,\nabla v)$ is Lipschitz continuous from $H^1(\Omega)$ into $L^2(\Omega)$. We will often abbreviate this composition by writing instead $H[\nabla v]\coloneqq H(\cdot,\nabla v)$ a.e.\ in $\Omega$.

\paragraph{Subdifferential of the Hamiltonian.}Given arbitrary sets $A$ and $B$, an operator $\mathcal{M}$ that maps each point $x\in A$ to a \emph{subset} of $B$ is called a \emph{set-valued map from $A$ to $B$}, and we write $\mathcal{M}:A\rightrightarrows B$. For the Hamiltonian given by~\eqref{eq:Hamiltonian_def} its point-wise  subdifferential with respect to $p$ is the set-valued map $ \partial_p H\colon\Omega\times\mathbb{R}^d\rightrightarrows\mathbb{R}^d$ defined by 
\begin{equation}\label{eq:H_subdifferential_def}
	\partial_p H(x,p)\coloneqq\left\{\tilde{b}\in\mathbb{R}^d:H(x,q)\geq H(x,p)+\tilde{b}\cdot(q-p),\quad\forall q\in\mathbb{R}^d\right\}.
\end{equation}
Note that $\partial_p H(x,p)$ is nonempty for all $x\in\Omega$ and $p\in\mathbb{R}^d$ because $H$ is real-valued and convex in $p$ for each fixed $x\in\Omega$. 
Note also that the subdifferential $\partial_p H$ is uniformly bounded since~\eqref{eq:Hamiltonian_Lipschitz} implies that for all $(x,p)\in\Omega\times \mathbb{R}^d$, the set $\partial_p H(x,p)$ is contained in the closed ball of radius $L_H=\lVert b \rVert_{C(\overline{\Omega}\times\mathcal{A};\mathbb{R}^d)}$ centred at the origin.

{We now define a set-valued mapping for measurable selections of the subdifferential when composed with the gradients of weakly differentiable functions.}
Given a function $v\in W^{1,1}(\Omega)$, we say that a real-valued vector field $\tilde{b}:\Omega\to\mathbb{R}^d$ is a \emph{measurable selection of $\partial_pH(\cdot,\nabla v)$} if $\tilde{b}$ is Lebesgue measurable and $\tilde{b}(x)\in \partial_p H(x,\nabla v(x))$ for a.e.\ $x\in \Omega.$
The uniform boundedness of the subdifferential sets implies that any measurable selection $\tilde{b}$ of $\partial_p H(\cdot,\nabla v)$ must belong to $L^\infty(\Omega;\mathbb{R}^d)$. 
Therefore, given $v\in W^{1,1}(\Omega)$ we can consider the set of all measurable selections of $\partial_p H(\cdot,\nabla v)$, which gives {rise to the set-valued map} $D_pH\colon W^{1,1}(\Omega)\rightrightarrows L^{\infty}(\Omega;\mathbb{R}^d)$ defined by
\begin{equation}
	D_pH[v]\coloneqq \left\{\tilde{b}\in L^{\infty}(\Omega;\mathbb{R}^d):\tilde{b}(x)\in\partial_pH(x,\nabla v(x)) \text{ \emph{for a.e.\ }}x\in\Omega\right\}.
\end{equation}
In \cite[Lemma 4.3]{osborne2022analysis} it was shown that $D_pH[v]$ is a nonempty subset of $L^{\infty}(\Omega;\mathbb{R}^d)$ for each $v$ in $ W^{1,1}(\Omega)$. 

{We also introduce} the set {of optimal feedback controls in $\mathcal{A}$, i.e.\ the set of controls that achieve the supremum in~\eqref{eq:Hamiltonian_def}.}
Define the set-valued map $\Lambda\colon\Omega\times\mathbb{R}^d\rightrightarrows \mathcal{A}$ via
\begin{equation}\label{eq:H_argmax_set}
	\Lambda(x,p)\coloneqq\text{argmax}_{\alpha\in\mathcal{A}}\{b(x,\alpha)\cdot p-f(x,\alpha)\}.
\end{equation}
Following \cite{osborne2022analysis}, for each $v\in W^{1,1}(\Omega)$, we define the set $\Lambda[v]$ of all Lebesgue measurable functions $\alpha^*:\Omega\to\mathcal{A}$ that satisfy $\alpha^*(x)\in\Lambda(x,\nabla v(x))$ for a.e.\ $x\in\Omega.$ We will refer to each element of $\Lambda[v]$ as \emph{a measurable selection of }$\Lambda(\cdot,\nabla v(\cdot))$. 
It is known that~$\Lambda[v]$ is nonempty for all $v\in W^{1,1}(\Omega)$, a result which ultimately rests upon the Kuratowski--Ryll-Nardzewski theorem~\cite{kuratowski1965general} on measurable selections, see also~\cite[Appendix~B]{smears2015thesis} for a detailed proof.
{In the later analysis, we} will make use of the following semismoothness result for the Hamiltonians, which was first shown in~\cite[Theorem~13]{smears2014discontinuous} already for the more general case of fully nonlinear second~order HJB operators, see also \cite[Lemma~7.1]{osborne2024regularization}.

\begin{lemma}[\cite{smears2014discontinuous,osborne2024regularization}]\label{lem:H_semismooth}
	Let $v\in H^1(\Omega)$ be given. For each $\epsilon>0$, there exists a $R>0$, depending only on $v$, $\epsilon$, $\Omega$, and $H$, such that 
	\begin{equation}\label{eq:H_semismooth}
		\sup_{\alpha\in {\Lambda[w]}}\left\|H[\nabla w] - H[\nabla v] - b(\cdot,\alpha)\cdot \nabla (w-v)\right\|_{H^{-1}(\Omega)} \leq \epsilon\|w - v\|_{H^1(\Omega)},
	\end{equation}
	for all $w\in H^1(\Omega)$ that satisfy $\|w- v\|_{H^1(\Omega)}\leq R$.
\end{lemma}
{We stress that the semismoothness result Lemma~\ref{lem:H_semismooth} does not require the Hamiltonian~$H$ to be differentiable. We also remark that it is essential in \eqref{eq:H_semismooth} that the supremum on the left-hand side ranges over $\alpha \in \Lambda[w]$, and it is not generally possible to replace this with $\alpha \in \Lambda[v]$ in the case of nondifferentiable Hamiltonians.}

{\subsection{Setting for discretization}\label{sec:setting_for_discretization}}

\paragraph{{Meshes and finite element spaces.}}
{For each $k\in\mathbb{N}$, let $\mathcal{T}_k$ denote a conforming simplicial mesh of the polytope ${\Omega}\subset\R^d$ (c.f.\ \cite[p.~51]{Ciarlet1978}). We assume that the sequence of meshes $\{\mathcal{T}_k\}_{k\in\mathbb{N}}$ is nested, in the sense that each element in $\mathcal{T}_{k+1}$ is either in $\Tk$ or is a subdivision of an element in $\Tk$. This assumption on nestedness of the sequence $\{\mathcal{T}_k\}_{k\in\mathbb{N}}$ will in particular play an implicit role in the analysis of this paper in order to apply some results from previous works, see Remark~\ref{rem:aprioriconvergence}. Given $k\in\mathbb{N}$ and an element $K\in \Tk$, we let $\diam K$ denote the diameter of $K$ w.r.t.\ the Euclidean distance on $\R^d$. The mesh-size function  $h_{\Tk}\in L^\infty(\Omega)$ associated with the mesh $\mathcal{T}_k$ is defined by  $h_{\Tk}|_K\coloneqq\diam K$ for each element $K\in\Tk$. For each $K\in\Tk$, we let $\rho_K$ denote the radius of the largest inscribed ball in $K$. In this paper we assume uniform shape-regularity of the sequence of meshes $\{\Tk\}_{k\in\N}$, which is to say that there exists a number $\delta>1$, independent of $k\in\mathbb{N}$, such that $h_{\Tk}|_K\leq \delta \rho_{K}$ for all $K\in\Tk$ and for all $k\in\N$. The mesh-size $h_k$ of $\Tk$ is defined as the maximum element diamter in $\Tk$: $h_k\coloneqq \norm{h_{\Tk}}_{L^\infty(\Omega)}$. We assume that $h_k\to 0$ as $k\to \infty$.

For each element $K\subset\R^d$, let $\mathcal{P}_1(K)$ denote the vector space of $d$-variate real-valued polynomials on $K$ of total degree at most one. Given $k\in\mathbb{N}$, we let 
\begin{equation}\label{eq:FEM_space_def}
	V_k \coloneqq \{v_k \in H^1_0(\Omega):\; v_k|_K \in \mathcal{P}_1(K) \quad \forall K\in\mathcal{T}_k \},
\end{equation}
which is the finite element space of $H^1_0$-conforming piecewise affine functions corresponding to the mesh $\Tk$. We endow the space $V_k$ with the standard norm on $H_0^1(\Omega)$, which we denote by $\|w_k\|_{V_k}\coloneqq \|w_k\|_{H^1(\Omega)}$ for $w_k \in V_k$. Note that the union $\cup_{k\in\mathbb{N}}V_k$ forms a dense subspace of $H_0^1(\Omega)$ since the mesh-size $h_k\to 0$ as $k\to\infty$.

For each $k\in \N$, we denote the set of all vertices of the mesh $\Tk$ by $\calVk$. We let $\calVki\coloneqq\calVk\cap \Omega$ denote the set of vertices of $\mathcal{T}_k$ that lie in the interior of the domain $\Omega$. We denote an enumeration of $\calVk$ by the set $\{x_i\}_{i=1}^{\Card\calVk}$, which we assume, without loss of generality, is ordered in such a way that $x_i\in \calVki$ if and only if $1\leq i\leq M_k\coloneqq \Card\calVki$. We then let the standard nodal Lagrange basis for the space $V_k$ be denoted by $\{ \xi_i \}_{i=1}^{M_k}$. The dependence on the mesh index $k$ is omitted from the notation of the nodal basis functions as there will be no risk of confusion in the upcoming analysis.}

Let $V_k^*$ denote the dual space of $V_k$, equipped with the dual norm $\norm{\cdot}_{V_k^*}$  defined by
\begin{equation}\label{eq:Vkstar_norm}
	\|\Psi \|_{V_k^*}\coloneqq \sup_{ v_k\in V_k \setminus \{0\} } \frac{\langle \Psi ,v_k\rangle_{V_k^*\times V_k}}{\norm{v_k}_{H^1(\Omega)} }\quad\forall \Psi \in V_k^*.
\end{equation}
Given an operator $L:V_k\to V_k^*$ we define its adjoint operator $L^*:V_k\to V_k^*$ by the pairing $\langle L^*w_k, v_k\rangle_{V_k^*\times V_k}\coloneqq\langle L v_k, w_k\rangle_{V_k^*\times V_k}$ for all $w_k,v_k\in V_{k}$.

\paragraph{Stabilization and discrete maximum principle.}

The class of stabilized FEM that we consider here is based on the principle of discretizing a general linear differential operator of the form $-\nu \Delta + \tilde{b}\cdot \nabla$, where $\tilde{b}\in L^\infty(\Omega;\R^d)$, by introducing a discrete linear operator $L\colon V_k\to V_k^*$ defined by
\begin{equation}\label{eq:discrete_lin_diff_op}
\begin{aligned}
\langle L w_k , v_k\rangle_{V_k^*\times V_k}\coloneqq (A_k \nabla w_k, \nabla v_k)_\Omega + (\tilde{b}\cdot \nabla w_k, v_k)_\Omega &&& \forall w_k,\,v_k \in V_k,
\end{aligned}
\end{equation}
where $A_k \in L^\infty\left(\Omega;\Rddsymp\right)$ is a numerical diffusion tensor of the form
\begin{equation}\label{eq:Ak_def}
	A_k \coloneqq \nu \mathbb{I}_d+D_k,
\end{equation}
with $\mathbb{I}_d$ denoting the $\dim\times\dim$ identity matrix and $D_k \in L^\infty\left(\Omega;\Rddsymp\right)$ is a chosen stabilization term. 
The following analysis is formulated in terms of two main assumptions on the choice of stabilization $D_k$.
The first main assumption guarantees the consistency of the stabilization, by requiring that the additional stabilization term involving $\Dk$ is vanishing in the small-mesh limit with optimal order:
\begin{enumerate}[label={(H\arabic*)},resume]
	\item The matrix-valued function $D_k\in L^\infty\left(\Omega;\Rddsymp\right)$ for all $k\in\mathbb{N}$ and there exists a constant $\CD$, independent of $k$, such that $\abs{D_k} \leq \CD h_{\Tk}$ a.e.\ in $\Omega$, for all $k\in\N$.\label{ass:bounded}
\end{enumerate}

The second main assumption that we state below requires that the stabilization should enforce a discrete maximum principle for a sufficiently large class of differential operators.
For each $k\in \N$, let $W(V_k,\Dk)$ denote the set of all linear operators $L:V_{k}\to V_{k}^*$ of the form~\eqref{eq:discrete_lin_diff_op} where $\tilde{b}\in L^\infty(\Omega;\R^d)$ satisfies $\|\tilde{b}\|_{L^{\infty}(\Omega;\R^d)}\leq L_H$, and where $A_k$ is given by~\eqref{eq:Ak_def}.
Recall that $L_H$ is the Lipschitz constant of the Hamiltonian, given in~\eqref{eq:Hamiltonian_Lipschitz} above.
Since we are considering here problems with homogeneous Dirichlet boundary conditions, a discrete linear operator $L:V_k\to V_k^*$ is said to satisfy the \emph{discrete maximum principle} (DMP) provided that the following condition holds: if $v_k\in V_{k}$ and $\langle Lv_k,\xi_i\rangle_{V_k^*\times V_k}\geq 0$ for all $ i\in\{1,\cdots,M_k\}$, then $v_k\geq 0$ in $\Omega$. 
{Recall that $\{\xi_i\}_{i=1}^{M_k}$ denotes the nodal basis functions for $V_k$.}
Note that a linear operator $L\colon V_k\to V_k^*$ satisfies the DMP if and only if its adjoint $L^*$ also satisfies the DMP. Moreover, if a linear operator $L\colon V_k\to V_k^*$ satisfies the DMP, then $L$ is invertible.
We now state the second main assumption on the choice of stabilization:
\begin{enumerate}[label={(H\arabic*)},resume]
	\item For every $k\in\mathbb{N}$, every $L\in W(V_k,\Dk)$ satisfies the discrete maximum principle.\label{ass:dmp}
\end{enumerate}
Under the above assumptions \ref{ass:bounded} and \ref{ass:dmp}, it is shown in~\cite[Lemma 3.1]{osborne2024near} that we have the following uniform stability bound for operators in the class $W(V_k,\Dk)$: there exists a constant $\Cstab>0$ that is independent of $k\in\mathbb{N}$, such that, for any $k\in \N$ and any $L\in W(V_k,D_k)$, 
\begin{subequations}
\begin{alignat}{2} 
\label{eq:infsup_stab_linop_1}
\norm{w_k}_{H^1(\Omega)} &\leq \Cstab \norm{L w_k}_{V_k^*}  &\quad&\forall w_k\in V_k,
\\ 
\label{eq:infsup_stab_linop_2}
\norm{w_k}_{H^1(\Omega)} &\leq \Cstab \norm{L^* w_k}_{V_k^*} &\quad&\forall w_k\in V_k,
\end{alignat}
\end{subequations}

\begin{remark}[Examples of stabilizations]
In practice, there is a range of possible stabilizations that satisfy the assumptions~\ref{ass:bounded} and \ref{ass:dmp}. Several examples are shown in~\cite{osborne2022analysis,osborne2023finite,osborne2024near}, including an isotropic stabilization on strictly acute meshes \cite{osborne2022analysis}, and the volume-based stabilization of \cite{osborne2023finite} for the more general case of meshes satisfying the Xu--Zikatanov condition~\cite{xu1999monotone}.
\end{remark}

{\paragraph{Notation for inequalities.}
{In order to avoid the proliferation of generic constants in the analysis,} we write $a\lesssim b$ for real numbers $a$ and $b$ if $a\leq C b$ for some constant $C$ that may depend on the problem data appearing in the MFG PDI \eqref{eq:mfg_pdi_sys} and some fixed parameters for the sequence of meshes $\{\Tk\}_{k\in\mathbb{N}}$ appearing below, such as the shape-regularity parameter $\delta$, but is otherwise independent of the regularization parameter $\lambda$ and the finite element mesh index $k\in\mathbb{N}$.}

\section{Weak formulation of MFG PDI and its discretization}\label{sec-eq:PDI_weakform-def-and-disc}

\subsection{Weak formulation of MFG PDI}
{We recall the definition of a weak solution of the MFG PDI~\eqref{eq:mfg_pdi_sys} that was introduced in~\cite{osborne2022analysis}, which is as follows:} \emph{A pair $(u,m)\in H_0^1(\Omega)\times H_0^1(\Omega)$ is a weak solution of~\eqref{eq:mfg_pdi_sys} if there exists a $\tilde{b}_*\in D_pH[u]$ such that}
\begin{subequations}\label{eq:PDI_weakform}
\begin{alignat}{2}
\nu(\nabla u,\nabla v)_\Omega+(H[\nabla u],v)_\Omega &=(F[m],v)_\Omega & \quad&\forall v\in H^1_0(\Omega), \label{eq:PDI_weakform_1}
\\ 
\nu(\nabla m,\nabla w)_\Omega+(m\tilde{b}_*,\nabla w)_\Omega &=\langle G,w\rangle_{H^{-1}\times H_0^1} & \quad&\forall w\in H^1_0(\Omega).\label{eq:PDI_weakform_2} 
\end{alignat} 
\end{subequations}
This definition only requires the existence of a suitable transport vector field $\tilde{b}_*\in D_pH[u]$, but its uniqueness is not required in general. {We refer the reader to~\cite[Section~3.3]{osborne2024thesis} for several examples showing that $\tilde{b}_*$ might be unique in some cases, but not in others.}

We recall that $G$ is nonnegative in the sense of distributions in $H^{-1}(\Omega)$, and we note that $F$ is strictly monotone on $H_0^1(\Omega)$. {Under the present conditions on the model data, c.f.\ Section \ref{sec-notation-fem-spaces}, Theorems 3.3 and 3.4 in \cite{osborne2022analysis,osborne2024erratum} imply that there exists a unique weak solution $(u,m)\in H_0^1(\Omega)\times H_0^1(\Omega)$ of the MFG PDI~\eqref{eq:mfg_pdi_sys}} which satisfies the bounds
\begin{subequations}\label{eq:PDI_solution_apriori_bounds}
\begin{align}
\|u\|_{H^1(\Omega)} &\lesssim 1+\|G\|_{H^{-1}(\Omega)}+\|f\|_{C({\overline{\Omega}\times\mathcal{A}})}, \label{eq:PDI_apriori_bounds_u}
\\
\|m\|_{H^1(\Omega)} &\lesssim \|G\|_{H^{-1}(\Omega)},\label{eq:PDI_apriori_bounds_m}
\end{align}
\end{subequations}
where the hidden constants depend only on $d$, $\Omega$, $\nu$, $L_H$,  and $L_F$.
{
Furthermore, it follows from the linear growth of $F$ and $H$ in \eqref{eq:F_linear_growth} and \eqref{eq:Hamiltonian_linear_growth}, respectively, and from the bounds above, that the HJB equation \eqref{eq:PDI_weakform_1} and the elliptic regularity bound~\eqref{eq:elliptic_regularity} imply in addition  $u\in H^{1+\gamma}(\Omega)$ with
\begin{equation}\label{eq:PDI_u_fractional_regularity}
\norm{u}_{H^{1+\gamma}(\Omega)} \lesssim 1 + \norm{G}_{H^{-1}(\Omega)} + \norm{f}_{C(\overline{\Omega}\times \mathcal{A})}.
\end{equation}
}

\subsection{Monotone finite element {discretization}}\label{subsec-FEM-mfg-pdi}

In order to accommodate a range of stabilized methods such as those in~\cite{osborne2023finite,osborne2022analysis,osborne2024erratum}, we consider the class of discretizations of the weak formulation~\eqref{eq:PDI_weakform} defined as follows: \emph{for each $k\in\mathbb{N}$, find $(u_k,m_k)\in V_{k}\times V_{k}$ such that there exists a $\tilde{b}_k\in D_pH[u_k]$ that satisfies}
\begin{subequations}\label{eq:pdi_fem_system}
\begin{alignat}{2}
(A_k\nabla u_k,\nabla v_k)_\Omega+(H[\nabla u_k],v_k)_\Omega &= ( F[m_k],v_k)_\Omega &\quad& \forall v_k\in V_k, \label{eq:pdi_fem_1}
\\
(A_k\nabla m_k,\nabla w_k)_\Omega+(m_k\tilde{b}_k,\nabla w_k)_\Omega &=\langle G,w_k\rangle_{H^{-1}\times H_0^1}  &\quad&\forall w_k\in V_k. \label{eq:pdi_fem_2}
\end{alignat}
\end{subequations}

\begin{remark}[Well-posedness of discrete approximations and convergence of \eqref{eq:pdi_fem_system}]\label{rem:aprioriconvergence}
{The existence and uniqueness of the numerical solutions is shown in~\cite[Theorems~5.2 \&~5.3]{osborne2022analysis,osborne2024erratum} for the case of strictly acute meshes. Following the arguments of the proofs in these results, it is straightforward to show that existence and uniqueness of the discrete solution also holds for the abstract class of methods considered here of the form~\eqref{eq:pdi_fem_system} under the assumptions~\ref{ass:bounded} and~\ref{ass:dmp}.}
{Moreover, in the current setting it is also straightforward to extend~\cite[Theorem 5.4, Corollary 5.5]{osborne2022analysis} on the convergence of the method to deduce} that $\{u_k\}_{k\in\mathbb{N}}$, $\{m_k\}_{k\in\mathbb{N}}$ are uniformly bounded in the $H^1$-norm with $u_k\to u$ in $H_0^1(\Omega)$ and $m_k\to m$ in $L^q(\Omega)$ where $q\in [1,\frac{2d}{d-2})$ as $k\to\infty$. 
{In particular, we have $m_k \to m$ in $L^2(\Omega)$ as $k\to \infty$ since $\frac{2d }{d-2} > 2$ for all $d \in \N$.}
{We note that we will use these results on plain convergence of the method in the analysis that comes below. Since the analysis from~\cite{osborne2022analysis,osborne2024erratum} assumes that the finite element spaces are nested, we therefore make the same assumption here.}
\end{remark}

{
As a first step towards bounding the error between the exact solution $(u,m)$ and its discrete approximations $(u_k,m_k)$, the following Lemma shows that the $H^1$-norm error $\norm{u-u_k}_{H^1(\Omega)}$ for the value function can be bounded in terms of the best-approximation error for $u$ from the finite element space, plus a consistency term that is related to the stabilization, and plus the $L^2$-norm error $\norm{m-m_{k}}_\Omega$ for the density function. 

\begin{lemma}[Bound on value function error in terms of the density function error]\label{lem:PDI_HJB_stability}
Suppose that~\ref{ass:bounded} and \ref{ass:dmp} hold. Then, for all $k$ sufficiently large,
\begin{equation}\label{eq:PDI_HJB_stability}
\|u - u_k\|_{H^1(\Omega)} \lesssim \inf_{\overline{u}_k\in V_k}\|u - \overline{u}_k\|_{H^1(\Omega)} + \|h_{\mathcal{T}_k}\nabla u\|_{\Omega} + \|m - m_k\|_{\Omega}.
\end{equation}
\end{lemma}
\begin{proof}
For each $k\in \N$, let $\alpha_k\colon \Omega\to \mathcal{A}$ be an arbitrary element of the set $\Lambda[u_k]$, where we recall that the set $\Lambda[u_k]$ is defined in Section~\ref{sec:model_data} above. 
Then, let $L_k\colon V_k\to V_k^*$ be the linear operator defined by $\langle L_k w_k , v_k\rangle_{V_k^*\times V_k}\coloneqq (A_k \nabla w_k,\nabla v_k)_\Omega+(\widetilde{b}_k \nabla w_k, v_k)_\Omega$ for all $w_k,\,v_k\in V_k$, where $\widetilde{b}_k(x)\coloneqq b(x,\alpha_k(x))$ for all $x\in\Omega$. 
Since $\norm{\widetilde{b}_k}_{L^\infty(\Omega,\R^d)}\leq L_H = \norm{b}_{C(\overline{\Omega}\times\mathcal{A})}$, it follows that $L_k \in W(V_k,\Dk)$, where it is recalled that the class $W(V_k,\Dk)$ is defined in Section~\ref{sec:setting_for_discretization} above.
Let $\overline{u}_k\in V_k$ be arbitrary.
It follows from~\eqref{eq:infsup_stab_linop_1} that $\norm{\overline{u}_k-u_k}_{H^1(\Omega)}\leq \Cstab \norm{ L_k (\overline{u}_k-u_k)}_{V_k^*}$. 
After considering the natural extension of the domain of $L_k$ to $H^1_0(\Omega)$, we add and subtract $\langle L_k u, v_k \rangle_{V_k^*\times V_k}$ for an arbitrary $v_k\in V_k$ to find that
\begin{equation}
\langle L_k (\overline{u}_k-u_k), v_k \rangle_{V_k^*\times V_k} = \langle L_k (\overline{u}_k-u),v_k\rangle_{V_k^*\times V_k}+ \langle L_k (u-u_k),v_k\rangle_{V_k^*\times V_k} \quad \forall v_k\in V_k.
\end{equation}
It is clear that $\abs{\langle L_k (\overline{u}_k-u),v_k\rangle_{V_k^*\times V_k}} \lesssim \norm{\overline{u}_k-u}_{H^1(\Omega)}\norm{v_k}_{H^1(\Omega)}$ for all $v_k\in V_k$. Then, using that $A_k$ is given by~\eqref{eq:Ak_def} and using the HJB equation~\eqref{eq:PDI_weakform_1} and its discrete counterpart~\eqref{eq:pdi_fem_1} (noting that $v_k\in V_k \subset H^1_0(\Omega)$ is a conforming discrete test function), we find that
\begin{multline}
\langle L_k (u-u_k),v_k\rangle_{V_k^*\times V_k}
 = (F[m]-F[m_k],v_k)_\Omega + (D_k \nabla u,\nabla v_k)_\Omega
	\\  + (H[\nabla u_k]-H[\nabla u]- \widetilde{b}_k\cdot\nabla(u_k-u),v_k)_\Omega,
\end{multline}
for all $v_k\in V_k$.
To bound the terms above, we first note that $|(F[m]-F[m_k],v_k)_\Omega|\leq L_F \norm{m-m_k}_\Omega \norm{v_k}_\Omega$ by Lipschitz continuity of $F$, c.f.\ \eqref{eq:F_lipschitz}.
Then, note that $|(D_k \nabla u,\nabla v_k)_\Omega| \leq  \CD \norm{h_{\Tk} \nabla u}_\Omega \norm{v_k}_{H^1(\Omega)}$ as a result of~\ref{ass:bounded}.
Recall from Remark~\ref{rem:aprioriconvergence} that $u_k\to u$ in $H^1_0(\Omega)$ as $k\to \infty$, and also that $\widetilde{b}_k = b(\cdot,\alpha_k)$ with $\alpha_k \in \Lambda[u_k]$ for each $k\in \N$.
Therefore, we can use Lemma~\ref{lem:H_semismooth} with $v=u$ and $\epsilon=\frac{1}{2\Cstab}$ to find that, for all $k$ sufficiently large,
\begin{equation}
\norm{ H[\nabla u_k]-H[\nabla u]- \widetilde{b}_k\cdot\nabla(u_k-u) }_{H^{-1}(\Omega)}\leq \frac{1}{2\Cstab}\norm{u-u_k}_{H^1(\Omega)}. 
\end{equation}
After using the triangle and the bounds above, it is then found that
\begin{equation}
\begin{split}
\norm{u-u_k}_{H^1(\Omega)} &\leq \norm{u-\overline{u}_k}_{H^1(\Omega)}+\norm{\overline{u}_k-u_k}_{H^1(\Omega)} \\
& \leq \norm{u-\overline{u}_k}_{H^1(\Omega)} + \Cstab \norm{L_k(\overline{u}_k-u_k)}_{V_k^*} 
\\ 
&\leq C \norm{u-\overline{u}_k} + \Cstab L_F\norm{m-m_k}_\Omega + \CD\Cstab \norm{h_{\Tk}\nabla u}_\Omega + \frac{1}{2}\norm{u-u_k}_{H^1(\Omega)},
\end{split}
\end{equation}
where the generic constant $C$ above depends only on $\Cstab$, $\nu$, $L_H$ and on $\CD$.
Therefore, we can absorb the final term on the right-hand side above up to a constant times the left-hand side, and we then conclude~\eqref{eq:PDI_HJB_stability} upon recalling that $\overline{u}_k \in V_k$ was arbitrary.
\end{proof}

Lemma~\ref{lem:PDI_HJB_stability} shows that in order to bound the joint error for both value and density functions, it remains only to bound the final term $\norm{m-m_k}_\Omega$ that appears on the right-hand side of \eqref{eq:PDI_HJB_stability}.

}

\section{Regularized continuous problem and its discretization}\label{sec:regularization}
In this section we introduce the regularized problems that will be employed in the \emph{a priori} error analysis.
The Moreau--Yosida regularization of the Hamiltonian and its main properties are summarized in Subsection~\ref{sec:moreau_yosida} below. 

\subsection{Moreau--Yosida Regularization of the Hamiltonian}\label{sec:moreau_yosida}

We consider here the Moreau--Yosida regularization of the Hamiltonian $H$ as follows: for each $\lambda\in (0,1]$, let $H_{\lambda}:\overline{\Omega}\times\mathbb{R}^d\to\mathbb{R}$ be defined by
\begin{equation}\label{eq:regularized_H_def}
	H_{\lambda}(x,p)\coloneqq \inf_{q\in\mathbb{R}^d}\left\{H(x,q) + \frac{1}{2\lambda }|q-p|^{2}\right\}.
\end{equation}
In the following, we will use some well-known elementary properties of Moreau--Yosida regularization, see e.g.\ \cite[Theorem 6.5.7]{aubinfrankowska1990setvalana}.
{First,} for each $\lambda\in (0,1]$, the function $H_{\lambda}$ is continuous on $\overline{\Omega}\times \R^\dim$ and additionally Lipschitz continuous in its second argument with
\begin{equation}\label{eq:reg_H_Lipschitz}
\begin{aligned}
|H_{\lambda}(x,p)-H_{\lambda}(x,q)|\leq L_H\lvert p-q\rvert &&&\forall(x,p,q)\in\overline{\Omega}\times\mathbb{R}^d\times\mathbb{R}^d,
\end{aligned}
\end{equation}
Note that the Lipschitz constant in~\eqref{eq:reg_H_Lipschitz} can be taken to be the same as the Lipschitz constant in the bound~\eqref{eq:Hamiltonian_Lipschitz} for $H$.
{Second,} for each $\lambda\in (0,1]$ and each $x\in\overline{\Omega}$, the function $\mathbb{R}^d\ni p\mapsto H_{\lambda}(x,p)$ is convex, and the partial derivative $\frac{\partial H_{\lambda}}{\partial p}:\overline{\Omega}\times\mathbb{R}^d\to\mathbb{R}^d$ exists, and is Lipschitz continuous w.r.t $p$, with
\begin{equation}\label{eq:reg_H_deriv_Lipschitz}
\left|\frac{\partial H_{\lambda}}{\partial p}(x,p)-\frac{\partial H_{\lambda}}{\partial p}(x,q)\right|\leq \frac{1}{\lambda}|p-q|\qquad\forall(x,p,q)\in\overline{\Omega}\times\mathbb{R}^d\times\mathbb{R}^d.
\end{equation}
{Third,} the error between $H$ and $H_\lambda$ satisfies the following inequality
\begin{equation}\label{eq:reg_H_approx}
\begin{aligned}
|H_{\lambda}(x,p) - H(x,p)|\leq {\frac{L_H^2\lambda }{2}}&& &\forall (x,p)\in \overline{\Omega}\times\mathbb{R}^d .
\end{aligned}
\end{equation} 
For a proof of~\eqref{eq:reg_H_approx}, see for instance~\cite{osborne2024regularization}.
Note that the bounds~\eqref{eq:reg_H_Lipschitz} and~\eqref{eq:reg_H_approx} imply the bound
\begin{equation}\label{eq:reg_H_linear_growth}
|H_{\lambda}(x,p)|
\leq \widetilde{C}_H\left(|p|+1\right)\quad \forall (x,p)\in \overline{\Omega}\times\mathbb{R}^d,
\end{equation}
for some constant $\widetilde{C}_H\geq 0$ independent of $\lambda\in (0,1]$, and 
{$\frac{\partial H_{\lambda}}{\partial p}$ is uniformly bounded with}
\begin{equation}\label{eq:reg_H_deriv_linf_bound}
\left|\frac{\partial H_{\lambda}}{\partial p}(x,p)\right|\leq L_H \quad \forall (x,p)\in \overline{\Omega}\times\mathbb{R}^d.
\end{equation}
Note also that the restriction of $\lambda$ to the interval $(0,1]$ is not essential, e.g.\ it can be replaced by some more general bounded set for which $0$ is a limit point.

\subsection{{Weak solutions of regularized problems}}
We now consider the weak formulation of the regularized problems~\eqref{eq:regularized_MFG} accompanied by homogeneous Dirichlet boundary conditions.
{For each $\lambda\in (0,1]$, let $({u}_{\lambda},{m}_{\lambda})\in H_0^1(\Omega)\times H_0^1(\Omega)$ denote the unique solution of}
\begin{subequations}\label{eq:regularized_weakform}
\begin{alignat}{2}
\nu(\nabla {u}_{\lambda},\nabla  v)_\Omega+(H_{\lambda}[\nabla {u}_{\lambda}],v)_\Omega 
 &= (F[{m}_{\lambda}],v)_\Omega  &\quad&\forall v\in H^1_0(\Omega),
\label{eq:regularized_weakform_1}
\\ 
\nu(\nabla {m}_{\lambda},\nabla w)_\Omega+\left({m}_{\lambda}\frac{\partial {H}_{\lambda}}{\partial p}[\nabla {u}_{\lambda}],\nabla w\right)_\Omega &=\langle {G},w\rangle_{H^{-1}\times H_0^1}  &\quad&\forall w \in H^1_0(\Omega).\label{eq:regularized_weakform_2} 
\end{alignat} 
\end{subequations}
{Note that the existence and uniqueness of $(u_\lambda,m_\lambda)$ is obtained, under the same hypotheses on the data as above, as a result of~\cite[Theorem 3.3 and 3.4]{osborne2022analysis}, which requires only the convexity and the Lipschitz continuity of the Hamiltonian with respect to the second variable, see also~\cite{osborne2024regularization}.}
Moreover, we have the following bounds on the norm of the solution, which are uniform with respect to the regularization parameter $\lambda$:
\begin{gather}
	\sup_{\lambda\in (0,1]}\|m_{\lambda}\|_{H^1(\Omega)} \lesssim \|G\|_{H^{-1}(\Omega)},\label{eq:regularized_apriori_bounds_m}
	\\
	{\sup_{\lambda\in (0,1]}\|u_{\lambda}\|_{H^1(\Omega)} \lesssim 1+\|G\|_{H^{-1}(\Omega)}+\|f\|_{C({\overline{\Omega}\times\mathcal{A}})}}, \label{eq:regularized_apriori_bounds_u}
\end{gather}
where the hidden constants depend only on $d$, $\Omega$, $\nu$, $L_H$,  and $L_F$. Furthermore, the density function $m_{\lambda}$ is nonnegative a.e.\ in $\Omega$ owing to the assumption that $G$ is nonnegative in the sense of distributions.

{Recall that we assume here that the source term $G$ has the form  $G=g_0-\nabla \cdot {g_1}$ with $g_0\in L^{q/2}(\Omega)$ and ${g_1}\in L^q(\Omega;\mathbb{R}^d)$ for some $q>d$. 
As a result, we obtain some uniform bounds in $L^\infty$ for the density functions. In particular, we deduce from~\cite[Theorem 8.15]{gilbarg2015elliptic} that the functions $m \in L^\infty(\Omega)$ and $m_{\lambda}\in L^\infty(\Omega)$, for all $\lambda\in (0,1]$, and that
\begin{equation}\label{eq:m_linfty_bound}
\|m\|_{L^{\infty}(\Omega)}+\sup_{\lambda\in (0,1]}\|m_{\lambda}\|_{L^{\infty}(\Omega)}\leq {\Mfty},
\end{equation} 
where {$\Mfty$ is a constant that depends only on $d$, $\Omega$, $L_H$,  $\nu$, $q$, $L_F$, $\norm{G}_{H^{-1}(\Omega)}$, $\norm{g_0}_{L^{q/2}(\Omega)}$ and $\norm{g_1}_{L^q(\Omega;\R^d)}$.} In particular, $\Mfty$ is independent of $\lambda$.

In addition, using the uniform linear growth of the regularized Hamiltonians $H_\lambda$,  c.f.~\eqref{eq:reg_H_linear_growth}, the linear growth of $F$, c.f.~\eqref{eq:F_linear_growth}, and the uniform bounds~\eqref{eq:regularized_apriori_bounds_m}and \eqref{eq:regularized_apriori_bounds_u} above, we also deduce from the regularized HJB equation \eqref{eq:regularized_weakform_1} and the elliptic regularity bound~\eqref{eq:elliptic_regularity} that $u_\lambda \in H^{1+\gamma}(\Omega)$ for all $\lambda \in (0,1]$ and that
\begin{equation}\label{eq:regularized_uniform_gamma}
\sup_{\lambda \in (0,1]}\norm{u_\lambda}_{H^{1+\gamma}(\Omega)} \lesssim 1+\|G\|_{H^{-1}(\Omega)}+\|f\|_{C({\overline{\Omega}\times\mathcal{A}})}, 
\end{equation}
where the hidden constant may depend on the problem data as above, including on $\gamma$, but is independent of $\lambda$.
}

\subsection{Convergence properties of regularized continuous problem}

{It was shown in~\cite{osborne2024regularization} that the solutions of the regularized problems converge to the solution of the MFG PDI as $\lambda \to 0$ in appropriate norms.
In addition, it was shown that for problems involving a strongly monotone coupling term $F$ and a nonnegative  source term $G$, as we are currently assuming, then the difference between solutions of the original problem and its regularizations enjoys an \emph{a priori} bound in terms of the regularization parameter $\lambda$.
In particular, the following result is found} in~\cite[Theorem 4.3]{osborne2024regularization}.
\begin{theorem}[\cite{osborne2024regularization}]\label{thm:pdi_reg_rate_lambda}
Let $(u,m)$ and $(u_{\lambda},m_{\lambda})$, $\lambda\in(0,1]$, be the respective unique solutions of~\eqref{eq:PDI_weakform} and~\eqref{eq:regularized_weakform}.
Then
\begin{equation}\label{eq:pdi_reg_rate_lambda}
\norm{u-u_{\lambda}}_{H^1(\Omega)}+\norm{m-m_{\lambda}}_{\Omega}\lesssim \lambda^{\frac{1}{2}},
\end{equation}
for all $\lambda$ sufficiently small, where the hidden constant depends only on $\Omega$, $\nu$, $\dim$, $L_H$, $L_F$, $c_F$ {and on $\norm{G}_{H^{-1}(\Omega)}$}.
\end{theorem}

{We emphasize that the bound in Theorem~\ref{thm:pdi_reg_rate_lambda} does not require or use any higher regularity of the solution $(u,m)$ beyond its minimal regularity in $H^1$.}
In the sequel we will show how the bound~\eqref{eq:pdi_reg_rate_lambda} will lead to quantitative control on the joint regularization error and discretization error in the finite element discretization of the regularized problem \eqref{eq:regularized_weakform}.

\subsection{Discretization of regularized problem and its convergence}\label{subsec-fem-reg-pb}

The finite element discretization of the regularized problem~\eqref{eq:regularized_weakform} is the following: for each $\lambda \in (0,1]$, find $(u_{k,\lambda},m_{k,\lambda})\in V_{k}\times V_{k}$ such that
\begin{subequations}\label{eq:regularized_fem_system}
\begin{alignat}{2}
(A_k\nabla u_{k,\lambda},\nabla v_k)_\Omega+(H_{\lambda}[\nabla u_{k,\lambda}],v_k)_\Omega  &= (F[m_{k,\lambda}],v_k)_\Omega &\quad&\forall v_k\in V_k, \label{eq:regularized_fem_1}
\\
(A_k\nabla m_{k,\lambda},\nabla w_k)_\Omega+\left(m_{k,\lambda}\frac{\partial H_{\lambda}}{\partial p}[\nabla u_{k,\lambda}],\nabla w_k\right)_\Omega  &=\langle G,w_k\rangle_{H^{-1}\times H^1_0} &\quad&\forall w_k\in V_k. \label{eq:regularized_fem_2}
\end{alignat}
\end{subequations}
{Note that, under the hypotheses \ref{ass:bounded} and \ref{ass:dmp},  the $C^{1,1}$-regularity of the regularized Hamiltonian $H_{\lambda}$ and the current assumptions on the model data imply the existence and uniqueness of the approximation $(u_{k,\lambda},m_{k,\lambda})\in V_k\times V_k$ for each $\lambda\in (0,1]$, see~\cite[Remark~3.3]{osborne2024near}.}
{In the following, we will consider sequences of approximations where both the regularization parameter and the mesh-size vanish jointly in the limit. To this end, for each $k\in \N$, let $\lambda_k\in (0,1]$ be the regularization parameter to be associated to the mesh $\Tk$, where we require that $\lambda_k\to 0$ as $k\to \infty$. 
We will then consider the discrete approximation $(u_{\klk},m_{\klk})\in V_k\times V_k$ that is the unique solution of~\eqref{eq:regularized_fem_system} for $\lambda=\lambda_k$. 
The next Lemma shows that taking the regularization parameter to vanish jointly with the mesh-size leads to convergence of the sequence of approximations $(u_{\klk},m_{\klk})\in V_k\times V_k$ to the solution $(u,m)$ of the MFG PDI weak formulation~\eqref{eq:PDI_weakform}.}
\begin{lemma}[Existence, Stability and Convergence of Regularized FEM]\label{lem:reg_discrete_convergence}
{Assume the hypotheses~\ref{ass:bounded} and~\ref{ass:dmp}.}
Then, for each $k\in\mathbb{N}$, there exists a unique solution $(u_{k,\lambda_k},m_{k,\lambda_k})$ {of} \eqref{eq:regularized_fem_system}, and we have the uniform bounds
\begin{subequations}\label{eq:reg_discrete_convergence_apriori_bounds}
\begin{align}
\sup_{k\in\mathbb{N}}\|m_{k,\lambda_k}\|_{H^1(\Omega)}&\lesssim \|G\|_{H^{-1}(\Omega)},\label{eq:reg_discrete_m_apriori_bounds}
\\ \sup_{k\in\mathbb{N}}\|u_{k,\lambda_k}\|_{H^1(\Omega)}&\lesssim 1+\|G\|_{H^{-1}(\Omega)}+\|f\|_{C(\overline{\Omega}\times\mathcal{A})}.\label{eq:reg_discrete_u_apriori_bounds}
\end{align}
\end{subequations}
Moreover, as $k\to\infty$,
\begin{align}\label{eq:reg_discrete_convergence}
u_{k,\lambda_k} \to u\quad\text{in}\; H_0^1(\Omega),\quad 
m_{k,\lambda_k}\rightharpoonup m\quad\text{in}\; H_0^1(\Omega),
\quad m_{k,\lambda_k}\to m\quad\text{in}\; L^q(\Omega),
\end{align}
for any $q\in [1,2^*)$, where $2^*\coloneqq \infty$ if $d=2$ and $2^*\coloneqq \frac{2d}{d-2}$ if $d\geq 3$, and for any $q\in[1,\infty]$ if $d=1$.
\end{lemma}
{Since this convergence result is due to a combination of arguments from \cite[Theorem 5.4]{osborne2022analysis} and \cite[Theorem 4.1]{osborne2024regularization}, for completeness we include its proof in Appendix \ref{appendix:pf-of-lemma-hk-lamk-vanish-jointly}.}

The following Lemma provides a bound between the exact and numerical solutions of the regularized problems.
\begin{lemma}[Bound between exact and discrete solutions of regularized problems]\label{lem:discrete_reg_l2_bound}
Assume \ref{ass:bounded}, and~\ref{ass:dmp}.  Let~$\gamma$ be as in~\eqref{eq:elliptic_regularity}. For each $k\in\mathbb{N}$, let $(u_{\lambda_k},m_{\lambda_k})$ and $(u_{k,\lambda_k},m_{k,\lambda_k})$ be the respective unique solutions of \eqref{eq:regularized_weakform} and \eqref{eq:regularized_fem_system}.
Then, there exists a $k_*\in \N$ such that
\begin{equation}\label{eq:discrete_reg_l2_bound}
\begin{split}
{\norm{u_{\lamk} - u_{\klk} }_{H^1(\Omega)}} + \|{m}_{\lambda_k}-m_{\klk}\|_{\Omega}
\lesssim \lambda_k^{\frac{1}{2}} + \lambda_k^{-1}h_k^{\gamma} \quad \forall k\geq k_*.
\end{split}
\end{equation}
\end{lemma}
Importantly, Lemma~\ref{lem:discrete_reg_l2_bound} quantifies the interplay between the spatial mesh-size and regularization parameter. 
We postpone the proof of Lemma~\ref{lem:discrete_reg_l2_bound} to Section \ref{sec-proof-of-key-lemma} below. 

We conclude this section with the following result which provides a quantitative rate of convergence between the density approximations generated by the original finite element scheme \eqref{eq:pdi_fem_system} and the scheme obtained with the regularized Hamiltonian \eqref{eq:regularized_fem_system}. 
\begin{lemma}[Bound between discrete solutions of PDI and regularized problems]\label{lem:pdi_reg_discrete_lambda_rate}
Assume \ref{ass:bounded} and \ref{ass:dmp}. {For each $k\in\mathbb{N}$, let $(u_k,m_k)$ and $(u_{k,\lambda_k},m_{k,\lambda_k})$ be the respective unique solutions of~\eqref{eq:pdi_fem_system} and~\eqref{eq:regularized_fem_system}.}
Then, for all $k\in\mathbb{N}$,
\begin{equation}\label{eq:pdi_reg_discrete_lambda_rate}
\|m_k - m_{k,\lambda_k}\|_{\Omega}\lesssim \lambda_k^{\frac{1}{2}}.
\end{equation}
where the hidden constant depends only on $\Omega$, $c_F$, $\Cstab$, $L_H$, {and $\|G\|_{H^{-1}(\Omega)}$}.
\end{lemma}
\begin{proof}
{Fix $k\in\mathbb{N}$. Test both~\eqref{eq:pdi_fem_1} and~\eqref{eq:regularized_fem_1} with $v_k = m_k - m_{k,\lambda_k}$ {and} subtract the resulting equations to obtain
\begin{equation}\label{u-eqn-diff}
	\begin{split}
		&\left(F[m_k] - F[m_{k,\lambda_k}],m_k - m_{k,\lambda_k}\right)_{\Omega}
		\\&\qquad=\left(A_k\nabla (u_k- u_{k,\lambda_k}), \nabla (m_k - m_{k,\lambda_k})\right)_{\Omega}+\left(m_k,H[\nabla u_k] - H_{\lambda_k}[\nabla u_{k,\lambda_k}]\right)_{\Omega}\\
		&\qquad\qquad\qquad\qquad\qquad\qquad\qquad\qquad+\left(m_{k,\lambda_k},H_{\lambda_k}[\nabla u_{k,\lambda_k}] - H[\nabla u_k]\right)_{\Omega}. 
	\end{split}
\end{equation} 
Then, test both~\eqref{eq:pdi_fem_2} and~\eqref{eq:regularized_fem_2} with $w_k = u_k - u_{k,\lambda_k}$, and subtract the resulting equations to get 
\begin{equation}\label{m-eqn-diff}
	\begin{split}
		&\left(A_k\nabla (m_k - m_{k,\lambda_k}),\nabla (u_k - u_{k,\lambda_k})\right)_{\Omega} 
		+\left(m_k\tilde{b}_k,\nabla (u_k - u_{k,\lambda_k})\right)_{\Omega} 
		\\
		&\qquad\qquad\qquad\qquad\qquad\qquad\qquad+ \left(m_{k,\lambda_k}\frac{\partial H_{\lambda_k}}{\partial p}[\nabla u_{k,\lambda_k}],\nabla (u_{k,\lambda_k} - u_k)\right)_{\Omega}= 0. 
	\end{split}
\end{equation} 
Using the symmetry of the matrix $A_k$ as implied by hypothesis \ref{ass:bounded}, we subtract~\eqref{m-eqn-diff} from~\eqref{u-eqn-diff} to thus obtain
\begin{equation}
	\begin{split}
		&\left(F[m_k] - F[m_{k,\lambda_k}],m_k - m_{k,\lambda_k}\right)_{\Omega}
		\\
		&\quad\quad\quad\quad\quad=\left(m_{k,\lambda_k},H_{\lambda_k}[\nabla u_{k,\lambda_k}] - H[\nabla u_k]  + \frac{\partial H_{\lambda_k}}{\partial p}[\nabla u_{k,\lambda_k}]\cdot\nabla (u_k - u_{k,\lambda_k})\right)_{\Omega} 
		\\
		&\quad\quad\quad\quad\quad\quad\quad\quad+\left(m_k,H[\nabla u_k] - H_{\lambda_k}[\nabla u_{k,\lambda_k}] + \tilde{b}_k\cdot\nabla (u_{k,\lambda_k} - u_k)\right)_{\Omega}.
	\end{split}
\end{equation} 
Since $G$ is nonnegative in the sense of distributions and hypothesis \ref{ass:dmp} ensures the Discrete Maximum Principle, we deduce that $m_k$ and $m_{k,\lambda_k}$ are both nonnegative everywhere in $\Omega$. This fact, together with the convexity of $H_{\lambda_k}$ w.r.t.\ $p$ and the definition of the inclusion $\tilde{b}_k\in D_pH[u_k]$, leads us to deduce that
\begin{multline}
			\left(m_{k,\lambda_k},H_{\lambda_k}[\nabla u_{k,\lambda_k}] - H[\nabla u_k]  + \frac{\partial H_{\lambda_k}}{\partial p}[\nabla u_{k,\lambda_k}]\cdot\nabla (u_k - u_{k,\lambda_k})\right)_{\Omega} 
			\\
			\leq  \left(m_{k,\lambda},H_{\lambda_k}[\nabla u_k]- H[\nabla u_k] \right)_{\Omega}
\end{multline}
and 
\begin{multline}
	\left(m_k,H[\nabla u_k] - H_{\lambda_k}[\nabla u_{k,\lambda_k}] + \tilde{b}_k\cdot\nabla (u_{k,\lambda_k} - u_k)\right)_{\Omega}\\
	\leq \left(m_k,H[\nabla u_{k,\lambda_k}] - H_{\lambda_k}[\nabla u_{k,\lambda_k}]\right)_{\Omega}.
\end{multline}
Consequently,
\begin{equation}\label{eq:m-lam-approx_1}
	\begin{split}
		\left(F[m_k] - F[m_{k,\lambda_k}],m_k - m_{k,\lambda_k}\right)_{\Omega}
		\leq 
		&\left(m_{k,\lambda},H_{\lambda_k}[\nabla u_k]- H[\nabla u_k] \right)_{\Omega}
		\\&\quad\quad\quad\quad+ \left(m_k,H[\nabla u_{k,\lambda_k}] - H_{\lambda_k}[\nabla u_{k,\lambda_k}]\right)_{\Omega}.
	\end{split}
\end{equation}
{We then combine~\eqref{eq:m-lam-approx_1} with the strong monotonicity of $F$ in \eqref{eq:F_strong_mono}, the uniform bound $|H_\lambda-H|\leq {\frac{L_H^2\lambda_k}{2}}$ of~\eqref{eq:reg_H_approx}, and the Cauchy--Schwarz inequality to obtain
	\begin{multline}\label{eq:m-lam-approx_2}
		\norm{m_k-m_{k,\lambda_k}}_{\Omega}^2 		\leq c_F^{-1}\left(F[m_k] - F[m_{k,\lambda_k}],m_k - m_{k,\lambda_k}\right)_{\Omega} \\
		\lesssim
		\left(\|m_{k,\lambda_k}\|_{L^2(\Omega)}+\|m_k\|_{L^2(\Omega)}\right){{\lambda_k}},
	\end{multline}
	where the hidden constant depends only on $c_F$, $L_H$, and $\Omega$.}
Using the uniform operator bound \eqref{eq:infsup_stab_linop_2}, we deduce from the KFP equations satisfied by $m_k$ and $m_{k,\lambda_k}$, respectively, that $ \|m_k\|_{L^2(\Omega)}\leq\Cstab\|G\|_{H^{-1}(\Omega)}$ and $ \|m_{k,\lambda_k}\|_{L^2(\Omega)}\leq\Cstab\|G\|_{H^{-1}(\Omega)}$. Therefore, 
\begin{equation}\label{mk-mklamk-bound}
	\|m_k - m_{k,\lambda_k}\|_{\Omega}^{2}\lesssim {{\lambda_k}},
\end{equation}
{where the hidden constant depends only on $\Omega$, $L_H$, $\Cstab$, $\|G\|_{H^{-1}(\Omega)}$, and $c_F$.}
By taking square root in \eqref{mk-mklamk-bound} and noting that $k\in\mathbb{N}$ was arbitrary, we obtain \eqref{eq:pdi_reg_discrete_lambda_rate}.}
\end{proof}

\section{Main result on rate of convergence of FEM for MFG PDI}\label{sec:main_results}

The main result of this work provides a rate of convergence for the finite element discretization \eqref{eq:pdi_fem_system} of the MFG PDI weak formulation \eqref{eq:PDI_weakform}. 
\begin{theorem}\label{thm:main_convergence_rate}
{Suppose that~\ref{ass:bounded} and~\ref{ass:dmp} hold.
Let $\gamma$ be as in~\eqref{eq:elliptic_regularity}. Let $(u,m)$ denote the unique solution of \eqref{eq:PDI_weakform} and, for each $k\in\mathbb{N}$, let $(u_k,m_k)$ denote the unique solution of \eqref{eq:pdi_fem_system}.
Then, there exists a $k_*\in \N$ such that
\begin{equation}\label{eq:main_convergence_rate}
\norm{u - u_{k}}_{H^1(\Omega)} + \|m - m_{k}\|_{\Omega}  \lesssim h_k^{\gamma/3} \quad \forall k\geq k_*.
\end{equation}
}
\end{theorem}

Note that in computational experiments so far, see e.g.\ \cite{osborne2022analysis,osborne2023finite,osborne2024thesis}, the rates of convergence for the error has been observed to be of optimal order $h_k$ in many cases. Thus, it is not currently known if the rate of order $h_k^{\gamma/3}$ given by~\eqref{eq:main_convergence_rate} is pessimistic or if it is sharp in a worst-case scenario.

{
\begin{proof}
Lemma~\ref{lem:PDI_HJB_stability} implies that, for all $k$ sufficiently large,
\begin{equation}
\norm{u-u_k}_{H^1(\Omega)} + \|m-m_k\|_{\Omega}
\lesssim  \inf_{\overline{u}_k\in V_k}\norm{u-\overline{u}_k}_{H^1(\Omega)} + \norm{h_{\mathcal{T}_k}\nabla u}_{\Omega} + \|m-m_k\|_{\Omega}.
\end{equation}
{
Recall that $u\in H^{1+\gamma}(\Omega)$ and that~\eqref{eq:PDI_u_fractional_regularity} gives a bound on $\norm{u}_{H^{1+\gamma}(\Omega)}$ in terms of the model data.
Thus, to bound the best approximation error $\inf_{\overline{u}_k\in V_k}\norm{u-\overline{u}_k}_{H^1(\Omega)}$, we may use the Scott--Zhang quasi-interpolant~\cite{scott1990finite} and the theory of interpolation spaces~\cite{AdamsFournier03} to find that $\inf_{\overline{u}_k\in V_k}\norm{u-\overline{u}_k}_{H^1(\Omega)}\lesssim h_k^{\gamma}\norm{u}_{H^{1+\gamma}(\Omega)}$, where the hidden constant depends only on $\Omega$, $\dim$, and the shape-regularity of the meshes, but is otherwise independent of $k$.}
Since $\gamma \in (0,1]$ entails that $h_k\lesssim h_k^{\gamma}$ and since $\norm{u}_{H^1(\Omega)}\lesssim \norm{u}_{H^{1+\gamma}(\Omega)}$, it follows that
\begin{equation}
\inf_{\overline{u}_k\in V_k}\norm{u-\overline{u}_k}_{H^1(\Omega)} + \norm{h_{\mathcal{T}_k}\nabla u}_{\Omega} \lesssim h_k^{\gamma},
\end{equation}
where the hidden constant may depend on the problem data, including $\norm{G}_{H^{-1}(\Omega)}$ and $\norm{f}_{C(\overline{\Omega}\times\mathcal{A})}$, but is otherwise independent of $k$.
Hence, for all $k$ sufficiently large,
\begin{equation}\label{eq:main_convergence_rate_1}
\norm{u-u_k}_{H^1(\Omega)} + \|m-m_k\|_{\Omega}
\lesssim  h_k^{\gamma} + \|m-m_k\|_{\Omega},
\end{equation}
We now show how to bound $\|m-m_k\|_{\Omega}$.
We apply the triangle inequality
\begin{equation}\label{eq:main_convergence_rate_2}
\|m-m_k\|_{\Omega}
\leq  \|m-m_{\lambda_k}\|_{\Omega} +\|m_{\lambda_k}-m_{\klk}\|_{\Omega}+\|m_{k,\lambda_k}-m_k\|_{\Omega},
\end{equation}
and we bound each term in turn. 
Theorem~\ref{thm:pdi_reg_rate_lambda} and Lemmas~\ref{lem:pdi_reg_discrete_lambda_rate} and~\ref{lem:discrete_reg_l2_bound} show that
\begin{subequations}\label{easy-m-bound}
\begin{gather}
\|m-m_{\lambda_k}\|_{\Omega}+\|m_{k,\lambda_k}-m_k\|_{\Omega}\lesssim \lambda_k^{\frac{1}{2}},
\\ 
\|{m}_{\lambda_k}-m_{\klk}\|_{\Omega}
\lesssim \lambda_k^{\frac{1}{2}} + \lambda_k^{-1}h_k^{\gamma},
\end{gather}
\end{subequations}
for all $k$ sufficiently large, where the hidden constants in the above bounds are independent of $k\in\mathbb{N}$. 
Combining the bounds above with~\eqref{eq:main_convergence_rate_1} shows that, for all $k$ sufficiently large,
\begin{equation}
\norm{u-u_k}_{H^1(\Omega)} + \|m-m_k\|_{\Omega}
\lesssim \lambda_k^{\frac{1}{2}} + \lambda_k^{-1}h_k^{\gamma} + h_k^{\gamma}.
\end{equation}
Since $\lambda_k^{-1}\geq 1$ for all $k\in\mathbb{N}$, we get $h_k^{\gamma}\leq  \lambda_k^{-1} h_k^{\gamma} $, and hence 
\begin{equation}\label{eq:main_convergence_rate_3}
\norm{u-u_k}_{H^1(\Omega)} + \|m-m_k\|_{\Omega} \lesssim \lambda_k^{\frac{1}{2}}  + \lambda_k^{-1} h_k^{\gamma}
\end{equation} for all $k\in\mathbb{N}$ sufficiently large. 
{Since the left-hand side is independent of the choice of $\lambda_k$, it is easy to see that the right-hand side is minimized for the choice $\lambda_k \eqsim h_k^{2\gamma/3}$, which yields the bound
\begin{equation}
\begin{split}
\norm{u-u_k}_{H^1(\Omega)} + \|m-m_k\|_{\Omega}
\lesssim h_k^{\gamma/3}
\end{split}
\end{equation}
for all $k\in\mathbb{N}$ sufficiently large. This completes the proof of the theorem.}
\end{proof}
}

{
In computational practice, it is often easier to solve the discrete system of equations~\eqref{eq:regularized_fem_system}, with regularization, than the unregularized inclusion problem~\eqref{eq:pdi_fem_system}.
Therefore, it is also of practical interest to consider the error between $(u,m)$ and $(u_{\klk},m_{\klk})$.

{
\begin{corollary}\label{cor:exact_to_discrete_reg}
Suppose that~\ref{ass:bounded} and~\ref{ass:dmp} hold.
Let $\gamma$ be as in~\eqref{eq:elliptic_regularity}.
Then, for any sequence of $\lambda_k\in (0,1]$ such that $\lambda_k\to 0$ as $k\to \infty$, there exists a $k_*\in \N$, such that
\begin{equation}
\norm{u-u_{\klk}}_{H^1(\Omega)} + \norm{m - m_{\klk}}_\Omega \lesssim \lambda_{k}^{\frac{1}{2}} + \lambda_k^{-1} h_k^\gamma \quad \forall k\geq k_*.  \label{eq:exact_to_discrete_reg_1}
\end{equation}
\end{corollary}
\begin{proof}
The bound~\eqref{eq:exact_to_discrete_reg_1} follows immediately from Theorem~\ref{thm:pdi_reg_rate_lambda}, Lemma~\ref{lem:discrete_reg_l2_bound}, and the triangle inequality.
\end{proof}
}
}

{Corollary~\ref{cor:exact_to_discrete_reg} thus} shows that one can combine discretization and regularization in order to obtain a discrete problem that is easier to solve in practice, whilst still retainining some quantitative control on the error. Note that choosing $\lambda_k $ to be of order $h_k^{2\gamma/3}$ results in an error of order $h_k^{\gamma/3}$, similar to the bound in Theorem~\ref{thm:main_convergence_rate}.

\section{Proof of Lemma~\ref{lem:discrete_reg_l2_bound}}\label{sec-proof-of-key-lemma}

In this section, we will give the proof of Lemma~\ref{lem:discrete_reg_l2_bound} after suitable preparation, and thereby complete the proof of Theorem~\ref{thm:main_convergence_rate}.

\subsection{Stability of discretized HJB equation}\label{sec:stability_discretized_HJB}

{In this section, we analyse the consistency and stability properties of the discretization of the HJB equations with regularized Hamiltonians.
To this end, we introduce the following discrete residual operators.}
For each $k\in\mathbb{N}$, let the operator $\RkH\colon V_k\times V_k\to V_k^*$ for the regularized HJB equation be defined by
\begin{equation}\label{eq:HJB_residual_def}
\langle \RkH(\overline{u}_k,\overline{m}_k),v_k\rangle_{V_k^*\times V_k} \coloneqq (F[\overline{m}_k],v_k)_{\Omega} - \left(A_k\nabla \overline{u}_k,\nabla v_k\right)_\Omega- \left( H_{\lambda_k}[\nabla \overline{u}_k],v_k\right)_\Omega,
\end{equation}
for all $\overline{m}_k,\overline{u}_k,v_k,w_k\in V_k$.
{Observe that the unique solution $(u_{k,\lambda},m_{k,\lambda})\in V_k\times V_k$ {of the} discrete problem \eqref{eq:regularized_fem_system} is a solution of $\RkH(\overline{u}_k,\overline{m}_k)=0$.}

{The following Lemma concerns the consistency of the discretized HJB equation, i.e.\ if $(\overline{u}_k,\overline{m}_k)$ is close to the exact solution $(u_{\lambda_k},m_{\lambda_k})$ of the regularized problem~\eqref{eq:regularized_weakform} then the residual norm $\norm{\RkH(\overline{u}_k,\overline{m}_k)}_{V_k^*}$ should be comparatively small, up to a consistency term coming from the stabilization that is first-order with respect to the mesh-size.}
\begin{lemma}\label{lem:discrete_reg_HJB_consistency}
Assume~\ref{ass:bounded}.
For all $k\in\N$, and all $(\overline{u}_k,\overline{m}_k)\in V_k\times V_k$, we have
\begin{align} 
\|\RkH(\overline{u}_k,\overline{m}_k)\|_{V_k^*}&\lesssim \|u_{\lambda_k}-\overline{u}_k\|_{H^1(\Omega)} + \|m_{\lambda_k}-\overline{m}_k\|_{\Omega} + \norm{{h_{\Tk}}\nabla u_{\lambda_k}}_{\Omega},\label{eq:discrete_reg_HJB_consistency} 
\end{align}
where the hidden constants depend only on $d$, $\Omega$, $\nu$, $L_F$, $L_H$, $\CD$, and $\Mfty$. 
\end{lemma}
\begin{proof}
{We start by showing~\eqref{eq:discrete_reg_HJB_consistency}. 
Recall that the dual norm $\norm{\cdot}_{V_k^*}$ is defined in~\eqref{eq:Vkstar_norm} above.}
Since $(u_{\lambda_k},m_{\lambda_k})$ solves~\eqref{eq:regularized_weakform} and {since $V_k\subset H^1_0(\Omega)$}, we have
\begin{multline}\label{eq:discrete_reg_HJB_consistency_1}
\langle \RkH(\overline{u}_k,\overline{m}_k),v_k\rangle_{V_k^*\times V_k}
=(F[\overline{m}_k]-F[m_{\lamk}],v_k)_\Omega - (A_k\nabla (\overline{u}_k-u_{\lamk}), \nabla v_k)_\Omega
\\ - (H_{\lamk}[\nabla \overline{u}_k]- H_{\lamk}[\nabla u_{\lambda_k}],v_k)_\Omega- (D_k\nabla u_{\lamk},\nabla v_k)_\Omega,
\end{multline}
for all $v_k\in V_k$.
{Note that in obtaining~\eqref{eq:discrete_reg_HJB_consistency_1} above, we have used the definition of~$A_k$ in~\eqref{eq:Ak_def} which implies that $(\nu \nabla u_{\lamk},\nabla v_k)_\Omega = (A_k \nabla u_{\lamk},\nabla v_k)_\Omega-(\Dk \nabla u_{\lamk},\nabla v_k)_\Omega$.}
The Lipschitz continuity of $F$ in~\eqref{eq:F_lipschitz} implies that~$|(F[\overline{m}_k]-F[m_{\lambda_k}],v_k)_\Omega|\leq L_F \norm{\overline{m}_k-m_{\lamk}}_\Omega \norm{v_k}_{H^1(\Omega)}$ for all $v_k\in V_k$.
The definition of $A_k$ in~\eqref{eq:Ak_def} and the bound on $D_k$ from~\ref{ass:bounded} then imply that $\abs{ (A_k\nabla (\overline{u}_k-u_{\lamk}),\nabla v_k)_\Omega}\lesssim \norm{u_{\lamk}-\overline{u}_k}_{H^1(\Omega)}\norm{v_k}_{H^1(\Omega)}$ for all $v_k\in V_k$, for some hidden constant that depends only on $d$, $\nu$, $\CD$ and on $\Omega$.
The Lipschitz continuity of $H_{\lamk}$ in~\eqref{eq:Hamiltonian_Lipschitz} implies that $\abs{(H_{\lamk}[\nabla u_{\lamk}]-H_{\lamk}[\nabla \overline{u}_k],v_k)_\Omega}\lesssim \norm{u_{\lamk}-\overline{u}_k}_{H^1(\Omega)}\norm{v_k}_{H^1(\Omega)} $ for all $v_k\in V_k$, {where in particular we stress that the hidden constant does not depend on $\lambda_k$.}
The hypothesis~\ref{ass:bounded} implies that $\abs{(\Dk\nabla u_{\lamk},\nabla v_k)_\Omega}\leq \CD \norm{ {h_{\Tk}} \nabla u_{\lamk}}_{\Omega}\norm{v_k}_{H^1(\Omega)}$ for all $v_k\in V_k$. 
Combining these bounds then yields~\eqref{eq:discrete_reg_HJB_consistency}.
\end{proof}

{In Lemma~\ref{lem:reg_HJB_discrete_stability} below, we consider the stability of the discretized HJB equations with regularlized Hamiltonians. The idea is to show that, on sufficiently fine meshes, the difference between a general discrete function $\overline{u}_k\in V_k$ and the value function approximation $u_{\klk}$ appearing in~\eqref{eq:regularized_fem_system} is bounded in terms of the  the residual $\RkH(\overline{u}_k,\overline{m}_k)$, where $\overline{m}_k \in V_k$, plus some additional terms related to the difference in density approximations $\overline{m}_k-m_{\klk}$. However, in order to obtain a bound with constants that are independent of~$\lambda_k$, we will use the semismoothness of the unregularized Hamiltonian $H$ about the exact value function $u$ in the analysis, c.f.~Lemma~\ref{lem:H_semismooth} above, which will lead to some additional terms on the right-hand side of the stability bound.}

\begin{lemma}[Stability of discrete HJB equations with regularized Hamiltonians]\label{lem:reg_HJB_discrete_stability}
Assume the hypotheses \ref{ass:bounded} and \ref{ass:dmp}. There exists a $k_*\in\mathbb{N}$ such that
\begin{equation}\label{eq:reg_HJB_discrete_stability}
\|\overline{u}_{k}-{u}_{k,\lambda_k}\|_{H^1(\Omega)}\lesssim\|\RkH(\overline{u}_k,\overline{m}_k)\|_{V_k^*}+\lambda_k^{\frac{1}{2}}+\| u_{\lambda_k}-\overline{u}_k\|_{H^1(\Omega)} + \|\overline{m}_k-{m}_{k,\lambda_k}\|_{\Omega}
\end{equation}
for all $k\geq k_*$ and all $(\overline{u}_k,\overline{m}_k)\in V_k\times V_k$. 
\end{lemma}

\begin{proof}
{For each $k\in \N$,} let $\alpha_k\in \Lambda[u_{k,\lambda_k}]$ be arbitrarily chosen, and let the operator $L_k\in W(V_k,D_k)$ be given by
\begin{equation}
\langle L_k w_k,v_k\rangle_{V_k^*\times V_k}
\coloneqq (A_k\nabla w_k,\nabla v_k)_\Omega+(b_k\cdot\nabla w_k,v_k)_\Omega\quad\forall w_k,v_k\in V_k,
\end{equation} 
{where $b_k (x) \coloneqq b(x,\alpha_k(x)) $ for $x\in \Omega$.}
{For each $k\in \N$, let $\overline{m}_k,\overline{u}_k\in V_k$ be fixed but arbitrary.}
{Then, it follows from the stability bound~\eqref{eq:infsup_stab_linop_1} that
\begin{equation}\label{eq:reg_HJB_discrete_stability_1}
\norm{\overline{u}_{k}-{u}_{k,\lambda_k}}_{H^1(\Omega)}\leq \Cstab\norm{L_k(\overline{u}_{k}-{u}_{k,\lambda_k})}_{V_k^*}
\end{equation}
Then, using~\eqref{eq:regularized_fem_1} and the definition of $\RkH$ in~\eqref{eq:HJB_residual_def},} we find that
\begin{multline}\label{eq:HJB_discrete_stability_2}
\langle L_k( \overline{u}_{k}-{u}_{k,\lambda_k}),v_k\rangle_{V_k^*\times V_k}
\\= (F[\overline{m}_k] - F[m_{k,\lambda_k}],v_k)_\Omega
-\langle \RkH(\overline{u}_k,\overline{m}_k),v_k\rangle_{V_k^*\times V_k} 
 \\+
\left( H_{\lambda_k}[\nabla u_{k,\lambda_k}] - H_{\lambda_k}[\nabla \overline{u}_k],v_k\right)_\Omega+ (b_k\cdot\nabla (\overline{u}_k-u_{\klk}),v_k)_\Omega,
\end{multline}
for all $v_k\in V_k$.
To bound the remaining terms in~\eqref{eq:reg_HJB_discrete_stability}, we make the additional addition and subtractions to find that, for all $v_k\in V_k$,
\begin{equation}\label{eq:HJB_discrete_stability_3}
\begin{split}
\left( H_{\lambda_k}[\nabla u_{k,\lambda_k}] - \right.&\left. H_{\lambda_k}[\nabla \overline{u}_k],v_k\right)_\Omega  + (b_k\cdot\nabla (\overline{u}_k-u_{\klk}),v_k)_\Omega
\\ &= (H[\nabla u_{k,\lambda_k}]-H[\nabla u] - b_k\cdot\nabla( u_{k,\lambda_k}-u),v_k)_\Omega 
\\ 
&\quad+  ( H_{\lambda_k}[\nabla u_{k,\lambda_k}] - H[\nabla u_{k,\lambda_k}],v_k)_\Omega  
+ (H[\nabla u] - H_{\lambda_k}[\nabla u],v_k)_{\Omega}
\\ &\quad+ (H_{\lambda_k}[\nabla u] -H_{\lambda_k}[\nabla \overline{u}_k],v_k)_\Omega
+ (b_k\cdot \nabla (\overline{u}_k-u),v_k)_\Omega,
\end{split}
\end{equation}
where {$u$ is the value function component of the solution of~\eqref{eq:PDI_weakform_1}.}
Observe that the approximation bound between $H$ and $H_\lambda$ in~\eqref{eq:reg_H_approx} implies that
\begin{equation}
\abs{( H_{\lambda_k}[\nabla u_{k,\lambda_k}] - H[\nabla u_{k,\lambda_k}],v_k)_\Omega} + 	\abs{(H[\nabla u] - H_{\lambda_k}[\nabla u],v_k)_{\Omega}} \lesssim \lambda_k \norm{v_k}_\Omega,
\end{equation}
for all $v_k\in V_k$, with some hidden constant depending only $L_H$ and on $\Omega$.
The Lipschitz continuity of $H_{\lambda_k}$, c.f.~\eqref{eq:reg_H_Lipschitz}, and the bound $\norm{b_k}_{L^\infty(\Omega;\R^d)}\leq \norm{b}_{C(\overline{\Omega}\times \mathcal{A};\R^d)}$ implies that
\begin{equation}
\abs{(H_{\lambda_k}[\nabla u] -H_{\lambda_k}[\nabla \overline{u}_k],v_k)_\Omega}+\abs{ (b_k\cdot \nabla (\overline{u}_k-u),v_k)_\Omega } \lesssim \norm{u-\overline{u}_k}_{H^1(\Omega)}\norm{v_k}_{H^1(\Omega)},
\end{equation}
for all $v_k\in V_k$, with some hidden constant depending only on $L_H$.
Since $\alpha_k\in \Lambda[u_{k,\lambda_k}]$ and $b_k(\cdot)=b(\cdot,\alpha_k(\cdot))$ {and since $u_{\klk}\to u$ as $k\to \infty$ in $H^1(\Omega)$ by Lemma~\ref{lem:reg_discrete_convergence},} we apply Lemma~\ref{lem:H_semismooth} to deduce that there exists a $k_*\in \N$ such that, for all $k\geq k_*$,
\begin{equation}\label{eq:HJB_discrete_stability_4}
\|H[\nabla u_{k,\lambda_k}] - H[\nabla u] - b_k\cdot\nabla (u_{k,\lambda_k}-u)\|_{H^{-1}(\Omega)}\leq \frac{1}{2\Cstab}\|u - u_{k,\lambda_k}\|_{H^1(\Omega)}.
\end{equation}
It then follows from~\eqref{eq:reg_HJB_discrete_stability}, Lipschitz continuity of $F$, c.f.~\eqref{eq:F_lipschitz}, {and the bounds above} that
\begin{multline}\label{eq:HJB_discrete_stability_5}
\norm{\overline{u}_{k}-{u}_{k,\lambda_k}}_{H^1(\Omega)}  \leq \Cstab \norm{L_k(\overline{u}_{k}-{u}_{k,\lambda_k})}_{V_k^*}
\\  \leq C \left(\norm{\RkH(\overline{u}_k,\overline{m}_k)}_{V_k^*} + \lambda_k + \norm{u-\overline{u}_k}_{H^1(\Omega)} + \norm{\overline{m}_k-m_{\klk}}_{\Omega} \right)
\\  + \frac{1}{2}\norm{u-u_{\klk}}_{H^1(\Omega)},
\end{multline}
for all $k$ sufficiently large, {where the generic constant $C$ depends only on $\Cstab$, $L_H$, $L_F$ and on $\Omega$.}
We now consider $\norm{u-u_{\klk}}_{H^1(\Omega)}$ and we use the triangle inequality
\begin{multline}\label{eq:HJB_discrete_stability_6}
\norm{u-u_{\klk}}_{H^1(\Omega)} \leq \norm{u-\overline{u}_k}_{H^1(\Omega)}+\norm{\overline{u}_k-u_{\klk}}_{H^1(\Omega)}
\\
\leq (C+1)\left(\norm{\RkH(\overline{u}_k,\overline{m}_k)}_{V_k^*} + \lambda_k + \norm{u-\overline{u}_k}_{H^1(\Omega)}+\norm{\overline{m}_k-m_{\klk}}_{\Omega} \right)
\\ +\frac{1}{2}\norm{u-u_{\klk}}_{H^1(\Omega)},
\end{multline}
for all $k$ sufficiently large, which, after simplification, implies that
\begin{equation}\label{eq:HJB_discrete_stability_7}
\|u - u_{k,\lambda_k}\|_{H^1(\Omega)}\lesssim\|\RkH(\overline{u}_k,\overline{m}_k)\|_{V_k^*}+\lambda_k +  \| u-\overline{u}_k\|_{H^1(\Omega)} + \|\overline{m}_k-{m}_{k,\lambda_k}\|_{\Omega}
\end{equation}
for all $k$ sufficiently large.
It then follows from~\eqref{eq:HJB_discrete_stability_5} that
\begin{equation}\label{eq:HJB_discrete_stability_8}
\|\overline{u}_{k}-{u}_{k,\lambda_k}\|_{H^1(\Omega)}\lesssim\|\RkH(\overline{u}_k,\overline{m}_k)\|_{V_k^*}+\lambda_k+\| u-\overline{u}_k\|_{H^1(\Omega)} +\|\overline{m}_k-{m}_{k,\lambda_k}\|_{\Omega},
\end{equation}
for all $k$ sufficiently large.
Since $\|u-u_{\lambda_k}\|_{H^1(\Omega)}\lesssim \lambda_k^{\frac{1}{2}}$ for sufficiently large $k\in\mathbb{N}$, as shown by Theorem~\ref{thm:pdi_reg_rate_lambda}, and $\lambda_k\to 0$ in $(0,1]$ as $k\to\infty$ by hypothesis, we apply the triangle inequality to the term $\norm{u-\overline{u}_k}_{H^1(\Omega)}$ in \eqref{eq:HJB_discrete_stability_8} to obtain the desired bound \eqref{eq:reg_HJB_discrete_stability}.
\end{proof}

\subsection{Stability of discretized KFP equation and nonnegative approximations of the density}

{We will now consider the consistency and stability properties of the discretized KFP equation with regularized Hamiltonians.}
Let the discrete residual operator $\RkF\colon V_k\times V_k\to V_k^*$ for the regularized KFP equation be defined by
\begin{multline}\label{eq:KFP_residual_def}
\langle \RkF(\overline{u}_k,\overline{m}_k),w_k\rangle_{V_k^*\times V_k}
\\  \coloneqq \langle G,w_k\rangle_{H^{-1}\times H^1_0} - \left(A_k\nabla \overline{m}_k , \nabla w_k\right)_\Omega - \left( \overline{m}_k\frac{\partial H_{\lambda_k}}{\partial p}[\nabla \overline{u}_k] , \nabla w_k\right)_\Omega \quad \forall w_k\in V_k,
\end{multline}
for all $(\overline{u}_k,\overline{m}_k)\in V_k\times V_k$.

\paragraph{Stability of discretized KFP equation.}

{Recall that $(u_{\klk},m_{\klk})$ denotes the unique solution of the discretized MFG system~\eqref{eq:regularized_fem_system} with regularized Hamiltonians.
The following Lemma shows the stability of the discretized KFP equation with respect to the density variable, when fixing the first argument of the residual operator $\RkF$ to be the discrete value function $u_{\klk}$.}
\begin{lemma}[Stability of discrete KFP equation]\label{lem:regularized_KFP_stability}
Assume~\ref{ass:bounded} and \ref{ass:dmp}.
Then, for all $k\in\N$, we have
\begin{equation}\label{eq:regularized_KFP_stability}
\norm{m_{\klk}-\overline{m}_k}_{H^1(\Omega)} \lesssim \norm{\RkF(u_{\klk},\overline{m}_k)}_{V_k^*} \quad \forall \,\overline{m}_k\in V_k,
\end{equation}
where the constant depends only on $\Cstab$.
\end{lemma}
{
\begin{proof}
Let $k\in\N $ and let $\overline{m}_k\in V_k$ be arbitrary. 
Then, the inf-sup stability bound \eqref{eq:infsup_stab_linop_2}, with $w_k\coloneqq m_{\klk} - \overline{m}_k$, gives
\begin{equation}
\begin{aligned}
&\norm{m_{\klk} - \overline{m}_k}_{H^1(\Omega)} \\
&\lesssim \sup_{\substack{w_k\in V_k:\\ \|w_k\|_{H^1(\Omega)}=1}}\left[ (A_k\nabla(m_{\klk} - \overline{m}_k), \nabla w_k)_\Omega + \left((m_{\klk} - \overline{m}_k){\frac{\partial H_{\lambda_k}}{\partial p}[\nabla u_{\klk}]},\nabla w_k\right)_\Omega \right]
\\ &=  \sup_{\substack{w_k\in V_k:\\ \|w_k\|_{H^1(\Omega)}=1}}\left[ 
\langle G,w_k\rangle_{H^{-1}\times H^1_0} - ( A_k\nabla\overline{m}_k , \nabla w_k)_\Omega - \left( \overline{m}_k\frac{\partial H_{\lambda_k}}{\partial p}[\nabla u_{\klk}], \nabla w_k\right)_\Omega \right]
\\ &= \sup_{\substack{w_k\in V_k:\\ \|w_k\|_{H^1(\Omega)}=1}} \langle \RkF(u_{\klk},\overline{m}_k),w_k\rangle_{V_k^*\times V_k} = \norm{\RkF(u_{\klk},\overline{m}_k)}_{V_k^*},
\end{aligned}
\end{equation}
where we have used the discrete KFP equation of \eqref{eq:regularized_fem_system} in the second line above.
\end{proof}}

\paragraph{Nonnegative approximations of the density.}

{It is well-known from the analysis of Lasry \& Lions~\cite{lasry2007mean} that the nonnegativity of the density function plays an important role in the uniqueness proof of the solution of the MFG system with monotone couplings. It is therefore natural that approximations of the density that remain nonnegative in the whole domain will play an important role in the current analysis on the rates of convergence.}
{We are therefore particularly interested in functions $\overline{m}_k \in\Vkp$ where $\Vkp$ is defined as} the set of functions in $V_k$ that are nonnegative in $\Omega$.
{In addition, to deal with the challenge of possibly limited regularity of the density functions $m$ and $m_{\lambda_k}$, we are particularly interested in choices of nonnegative approximations $\overline{m}_k \in \Vkp$ that enable a consistency bound for the discrete residual $\RkF$ that involves only the $L^2$-norm of the difference $m_{\lamk}-\overline{m}_k$ to appear on the right-hand side.}
{This motivates the following choice of approximation: let} $\mkp \in V_k$ be the unique solution of
\begin{equation}\label{eq:mkp_def}
(A_k\nabla \mkp,  \nabla w_k)_\Omega+\left( \mkp\frac{\partial H_{\lambda_k}}{\partial p}[\nabla u_{\lambda_k}] , \nabla w_k\right)_\Omega =\langle G,w_k\rangle_{H^{-1}\times H_0^1}\quad\forall w_k\in V_k.
\end{equation}
{Observe that it is the continuous value function $u_{\lambda_k}$ from~\eqref{eq:regularized_weakform} that serves as the argument to the partial derivative of the regularized Hamiltonian $H_{\lamk}$ in~\eqref{eq:mkp_def} above.}
{Under the hypotheses~\ref{ass:bounded} and~\ref{ass:dmp}, the linear operator acting on $V_k$ that is appearing in~\eqref{eq:mkp_def} is in the class $W(V_k,\Dk)$, so it follows that $\mkp$ is well-defined in $V_k$ for each $k\in\N$.}
Moreover, the DMP implies that $\mkp \in \Vkp$ since $G$ is nonnegative in the sense of distributions.
It is also straightforward to show from~\eqref{eq:infsup_stab_linop_2} that {{$\mkp$} satisfies $\norm{\mkp}_{H^1(\Omega)}\lesssim \norm{G}_{H^{-1}(\Omega)}$ for some constant that depends only on $\Omega$, $\nu$, $L_H$, and $\Cstab$.}

{We now show two main properties of the discrete function $\mkp$ defined above.
The first result in Lemma~\ref{lem:KFP_res_consistency} below shows that the approximation $\mkp$ defined above yields a consistency bound for the residual operator $\RkF$ that involves the $L^2$-norm of the difference $m_{\lamk}-\mkp$ in the right-hand side, rather than the $H^1$-norm that would otherwise appear for a generic choice of $\overline{m}_k$.
The second result, in Lemma~\ref{lem:mkstar_l2_approx} below, then gives an \emph{a priori} bound on the $L^2$-norm error $\norm{m_{\lamk}-\mkp}_{\Omega}$ with respect to the mesh-size $h_k$ that does not require higher-regularity of $m_{\lamk}$ beyond minimal regularity in $H^1$.}
\begin{lemma}\label{lem:KFP_res_consistency}
Assume~\ref{ass:bounded} and~\ref{ass:dmp}. 
Then
\begin{equation}\label{eq:KFP_res_consistency}
\|\RkF(\overline{u}_k,\mkp)\|_{V_k^*}
\lesssim \lambda_k^{-1}\norm{u_{\lambda_k}-\overline{u}_k}_{H^1(\Omega)} + \norm{m_{\lambda_k}-\mkp}_{\Omega} \quad \forall \, \overline{u}_k\in V_k,
\end{equation}
where the hidden constant depends only on $L_H$ and $\Mfty$.
\end{lemma}
\begin{proof}
Fix $k\in\mathbb{N}$ and let $\overline{u}_k\in V_k$ be given. 
For any $w_k \in V_k$, it follows from the definition of $\mkp$ given in~\eqref{eq:mkp_def} that, for any $w_k\in V_k$,
\begin{equation}
\begin{aligned}
 &\langle \RkF(\overline{u}_k,\mkp),w_k\rangle_{V_k^*\times V_k}
\\ 
 &=\langle G,w_k\rangle_{H^{-1}\times H_0^1} - ( A_k\nabla \mkp, \nabla w_k)_\Omega-\left( \mkp\frac{\partial H_{\lambda_k}}{\partial p}[\nabla \overline{u}_k] , \nabla w_k \right)_\Omega  
\\
 &= \left( \mkp \left(\frac{\partial H_{\lambda_k}}{\partial p}[\nabla u_{\lambda_k}] - \frac{\partial H_{\lambda_k}}{\partial p}[\nabla \overline{u}_k]  \right) , \nabla w_k \right)_\Omega  
\\
 &= \left(  m_{\lambda_k}\left(\frac{\partial H_{\lambda_k}}{\partial p}[\nabla u_{\lambda_k}] - \frac{\partial H_{\lambda_k}}{\partial p}[\nabla \overline{u}_k]\right) , \nabla w_k\right)_\Omega 
  \\ & \qquad\qquad\qquad\qquad+
  \left(  (\mkp-m_{\lambda_k})\left(\frac{\partial H_{\lambda_k}}{\partial p}[\nabla u_{\lambda_k}] - \frac{\partial H_{\lambda_k}}{\partial p}[\nabla \overline{u}_k]\right),\nabla w_k\right)_\Omega .
\end{aligned}
\end{equation}
Consequently, it is found that
\begin{multline}
\norm{\RkF(\overline{u}_k,\mkp)}_{V_k^*}
\\
\leq {M_{\infty}} \left\|\frac{\partial H_{\lambda_k}}{\partial p}[\nabla u_{\lambda_k}] - \frac{\partial H_{\lambda_k}}{\partial p}[\nabla \overline{u}_k]\right\|_{\Omega} + 2L_H\|\mkp-m_{\lambda_k}\|_{\Omega} ,
\end{multline}
where we have applied \eqref{eq:reg_H_deriv_linf_bound} and the fact that $m_{\lambda_k}\in L^{\infty}(\Omega)$ by \eqref{eq:m_linfty_bound}. Then the Lipschitz property \eqref{eq:reg_H_deriv_Lipschitz} implies that
\begin{equation}
\|\RkF(\overline{u}_k,\mkp)\|_{V_k^*}\leq  {\Mfty} \lambda_k^{-1}\|u_{\lambda_k}-\overline{u}_k\|_{H^1(\Omega)} + 2L_H\|m_{\lambda_k}-\mkp\|_{\Omega},
\end{equation}
which shows~\eqref{eq:KFP_res_consistency}.
\end{proof}
\begin{lemma}\label{lem:mkstar_l2_approx}
{Assume~\ref{ass:bounded} and~\ref{ass:dmp}.}
Then, 
\begin{equation}\label{eq:mkstar_l2_approx}
\|m_{\lambda_k}-\mkp\|_{\Omega}\lesssim h_k^{\gamma} \quad \forall k\in \N.
\end{equation}
\end{lemma}
{
\begin{proof}
The proof follows the duality argument of \cite[Lemma 8.2]{osborne2024near}, so we only sketch the main ideas of the proof. As such, for each $k\in\mathbb{N}$, let $z_k$ be the unique function in $H_0^1(\Omega)$ such that 
\begin{equation}\label{eq:mkstar_l2_approx_1}
\nu(\nabla z_k, \nabla w)_\Omega+\left(\frac{\partial H_{\lambda_k}}{\partial p}[\nabla u_{\lambda_k}]\cdot\nabla z_k, w\right)_\Omega = (m_{\lambda_k}-\mkp,w)_\Omega \quad\forall w\in H_0^1(\Omega).
\end{equation}
Note that the existence and uniqueness of $z_k$ solving \eqref{eq:mkstar_l2_approx_1} for each $k\in\mathbb{N}$ is guaranteed by {\cite[Theorem~8.3]{gilbarg2015elliptic}}. Moreover $\norm{z_k}_{H^1(\Omega)}\lesssim \norm{m_{\lambda_k}-\mkp}_{\Omega}$ with a hidden constant depending only on $d$, $\nu$, $L_H$, and on $\Omega$, see also~\cite[Lemma~4.5]{osborne2022analysis,osborne2024erratum}.
Since the vector field $\frac{\partial H_{\lambda_k}}{\partial p}[\nabla u_{\lambda_k}]$ is in $ L^\infty(\Omega;\R^\dim)$, and we have that $m_{\lambda_k}-\mkp\in L^2(\Omega)$, we deduce from~{the elliptic regularity bound~\eqref{eq:elliptic_regularity}} that $z_k\in H^{1+\gamma}(\Omega)$.
Furthermore, we have the norm bound $\norm{z_k}_{H^{1+\gamma}(\Omega)} \lesssim\norm{m_{\lambda_k}-\mkp}_{\Omega}$. {We then follow the remainder of the duality argument of \cite[Lemma 8.2]{osborne2024near} to conclude the proof of the lemma. Indeed, since 
\begin{equation}\label{eq:mkstar_l2_approx_2}
\norm{z_k-I^{\mathrm{SZ}}_k {z}_k}_{H^1(\Omega)}\lesssim h_k^{\gamma}\norm{z_k}_{H^{1+\gamma}(\Omega)}, \quad \norm{I^{\mathrm{SZ}}_k {z}_k}_{H^1(\Omega)}\lesssim \norm{z_k}_{H^1(\Omega)},
\end{equation} where $I^{\mathrm{SZ}}_k {z}_k\in V_k$ denotes the Scott--Zhang quasi-interpolant of $z_k$ \cite{scott1990finite}, we obtain that $\norm{z_k-I^{\mathrm{SZ}}_k{z}_k}_{H^1(\Omega)}\lesssim h_k^{\gamma}\norm{m_{\lambda_k}-\mkp}_{\Omega}$ and $\norm{I^{\mathrm{SZ}}_k{z}_k}_{H^1(\Omega)}\lesssim\norm{m_{\lambda_k}-\mkp}_{\Omega}$. Then, taking $w=m_{\lambda_k}-\mkp$ in \eqref{eq:mkstar_l2_approx_1}, and using the definition of~$\mkp$ from~\eqref{eq:mkp_def}, we eventually get
\begin{multline}\label{eq:duality_1}
	\norm{m_{\lambda_k}-\mkp}_\Omega^2   =  \nu \left(\nabla (z_k-I^{\mathrm{SZ}}_k{z}_k),\nabla(m_{\lambda_k}-\mkp)\right)_{\Omega}\\+\left(m_{\lambda_k}-\mkp,\frac{\partial H}{\partial p}[\nabla u_{\lambda_k}]\cdot\nabla (z_k-I^{\mathrm{SZ}}_k{z}_k)\right)_{\Omega}
	\\+\left(\Dk\nabla\mkp,\nabla I^{\mathrm{SZ}}_k{z}_k\right)_{\Omega}.
\end{multline}
Combining~\eqref{eq:duality_1} with~\eqref{eq:mkstar_l2_approx_2}, we obtain
\begin{multline}\label{eq:duality_3}
	\norm{m_{\lambda_k}-\mkp}_\Omega^2 \\
	\lesssim  h_k^{\gamma}\norm{m_{\lambda_k}-\mkp}_{H^1(\Omega)}\norm{z_k}_{H^{1+\gamma}(\Omega)}+
	\norm{\Dk\nabla \mkp}_\Omega\norm{m_{\lambda_k}-\mkp}_{\Omega}
	\\ \lesssim h_k^{\gamma}\left(\norm{m_{\lambda_k}}_{H^1(\Omega)}+\norm{\mkp}_{H^1(\Omega)}\right) \norm{m_{\lambda_k}-\mkp}_{\Omega}
\end{multline}
where in the last line we used the bound on $\Dk$ from~\ref{ass:bounded} and the elliptic regularity of $z_k$ above. Simplifying then gives
\begin{equation}
	\|m_{\lambda_k}-\mkp\|_{\Omega}\lesssim h_k^{\gamma} \left(\norm{m_{\lambda_k}}_{H^1(\Omega)}+\norm{\mkp}_{H^1(\Omega)}\right)
\end{equation}
The stability bounds \eqref{eq:regularized_apriori_bounds_m} and $\norm{\mkp}_{H^1(\Omega)}\lesssim \norm{G}_{H^{-1}(\Omega)}$ then yield \eqref{eq:mkstar_l2_approx}.}
\end{proof}}

\subsection{Proof of Lemma~\ref{lem:discrete_reg_l2_bound}}

By following the argument in \cite[Lemma 6.2]{osborne2024near} and the existence of discrete approximations to \eqref{eq:regularized_fem_system} by Lemma \ref{lem:reg_discrete_convergence}, we have the $L^2$-norm stability in the density component for the regularized discrete MFG system~\eqref{eq:regularized_fem_system}, when restricted to the set of nonnegative functions in~$\Vkp$. {
\begin{lemma}\label{lem:L2_strong_mono_bound}
Assume~\ref{ass:dmp}. 
For each $k\in\mathbb{N}$, any $\overline{u}_k\in V_k$ and $\overline{m}_k\in \Vkp$, we have
\begin{multline}\label{eq:L2_strong_mono_bound}
c_F\|\overline{m}_k-m_{\klk}\|_{\Omega}^2\\
\leq 
\langle \RkH(\overline{u}_k,\overline{m}_k),\overline{m}_k-m_{\klk}\rangle_{V_k^*\times V_k} - \langle \RkF(\overline{u}_k,\overline{m}_k),\overline{u}_k-u_{\klk}\rangle_{V_k^*\times V_k}.
\end{multline}
\end{lemma}}
{
\begin{proof}
	Let $\overline{u}_k\in V_k$ and $\overline{m}_k\in \Vkp$ be fixed but arbitrary.
	To abbreviate the notation, let
	\begin{equation}
		\Rk\coloneqq \langle {R}_{k,\lambda_k}^1(\overline{u}_k,\overline{m}_k),\overline{m}_k-m_{k,\lambda_k}\rangle_{V_k^*\times V_k} - \langle {R}_{k,\lambda_k}^2(\overline{u}_k,\overline{m}_k),\overline{u}_k-u_{k,\lambda_k}\rangle_{V_k^*\times V_k}.
	\end{equation} 
	Since $R_{k,\lambda_k}^1({u}_{k,\lambda_k},{m}_{k,\lambda_k})=0$ and $R_{k,\lambda_k}^2({u}_{k,\lambda_k},{m}_{k,\lambda_k})=0$ in $V_k^*$, we have that
	\begin{multline}\label{diff-expr}
			\Rk=\langle {R}_{k,\lambda_k}^1(\overline{u}_k,\overline{m}_k) - R_{k,\lambda_k}^1({u}_{k,\lambda_k},{m}_{k,\lambda_k}),\overline{m}_k-m_{k,\lambda_k}\rangle_{V_k^*\times V_k} \\
			- \langle {R}_{k,\lambda_k}^2(\overline{u}_k,\overline{m}_k) - R_{k,\lambda_k}^2({u}_{k,\lambda_k},{m}_{k,\lambda_k}),\overline{u}_k-u_{k,\lambda_k}\rangle_{V_k^*\times V_k}.
	\end{multline} 
	Using the discrete regularized HJB equation~\eqref{eq:regularized_fem_1} and the discrete regularized KFP equation~\eqref{eq:regularized_fem_2} we find that
		\begin{multline}\label{diff-expr-hjb}
			\langle {R}_{k,\lambda_k}^1(\overline{u}_k,\overline{m}_k) - R_{k,\lambda_k}^1({u}_{k,\lambda_k},{m}_{k,\lambda_k}),\overline{m}_k-m_{k,\lambda_k}\rangle_{V_k^*\times V_k}
			\\
			=\left(F[\overline{m}_k] - F[m_{k,\lambda_k}],\overline{m}_k-m_{k,\lambda_k}\right)_{\Omega} +\left(H_{\lambda_k}[\nabla u_{k,\lambda_k}] - H_{\lambda_k}[\nabla \overline{u}_k],\overline{m}_k-m_{k,\lambda_k}\right)_{\Omega}
			\\+\left(A_k\nabla (u_{k,\lambda_k}-\overline{u}_k),\nabla (\overline{m}_k-m_{k,\lambda_k})\right)_{\Omega},
		\end{multline}
		and 
		\begin{multline}\label{diff-expr-kfp}
			\langle {R}_{k,\lambda_k}^2(\overline{u}_k,\overline{m}_k) - R_{k,\lambda_k}^2({u}_{k,\lambda_k},{m}_{k,\lambda_k}),\overline{u}_k-u_{k,\lambda_k}\rangle_{V_k^*\times V_k} \\
			= \left(m_{k,\lambda_k}\frac{\partial H_{\lambda_k}}{\partial p}[\nabla u_{k,\lambda_k}]- \overline{m}_k\frac{\partial H_{\lambda_k}}{\partial p}[\nabla \overline{u}_k],\nabla (\overline{u}_k-u_{k,\lambda_k})\right)_{\Omega}\\+\left(A_k\nabla (\overline{m}_k-m_{k,\lambda_k}),\nabla (u_{k,\lambda_k}-\overline{u}_k)\right)_{\Omega}.
		\end{multline}
		Using the symmetry of $A_k$ we then obtain from \eqref{diff-expr}, \eqref{diff-expr-hjb} and \eqref{diff-expr-kfp} that
	\begin{multline}\label{eq:L2mono_bound_1}
		\Rk=\left(F[\overline{m}_k] - F[m_{k,\lambda_k}],\overline{m}_k-m_{k,\lambda_k}\right)_{\Omega}
		\\ + (m_{k,\lambda_k}, R_{H_{\lambda_k}}[\nabla\overline{u}_k,\nabla u_{k,\lambda_k}])_\Omega + (\overline{m}_k,R_{H_{\lambda_k}}[\nabla u_{k,\lambda_k},\nabla \overline{u}_k])_\Omega,
	\end{multline}
	where we define $R_{H_{\lambda_k}}[\nabla v,\nabla w]\coloneqq H_{\lambda_k}[\nabla v]-H_{\lambda_k}[\nabla w]-\frac{\partial H_{\lambda_k}}{\partial p}[\nabla w]\cdot \nabla(v-w)$ for any $v,\,w \in H^1(\Omega)$.
	Since $H_{\lambda_k}$ is convex and differentiable w.r.t.\ $p$, we have $R_{H_{\lambda_k}}[\nabla v,\nabla w]\geq 0$ a.e.\ for any $v,\,w\in H^1(\Omega)$. Also, we have $m_{k,\lambda_k}\geq 0$ in $\Omega$ from \ref{ass:dmp} and $\overline{m}_{k}\geq 0$ in $\Omega$ by hypothesis.
	Therefore the terms $(m_{k,\lambda_k}, R_{H_{\lambda_k}}[\nabla\overline{u}_k,\nabla u_{k,\lambda_k}])_\Omega$ and $(\overline{m}_k,R_{H_{\lambda_k}}[\nabla u_{k,\lambda_k},\nabla \overline{u}_k])_\Omega$ in~\eqref{eq:L2mono_bound_1} are both nonnegative, and thus the strong monotonicity of $F$, c.f.~\eqref{eq:F_strong_mono}, implies that
	\begin{equation}\label{eq:L2mono_bound_2}
		c_F\norm{\overline{m}_k-m_{k,\lambda_k}}_\Omega^2 \leq\left(F[\overline{m}_k] - F[m_{k,\lambda_k}],\overline{m}_k-m_{k,\lambda_k}\right)_{\Omega} \leq \Rk,
	\end{equation}
	which shows~\eqref{eq:L2_strong_mono_bound}.
\end{proof}}

{In the next step, we combine the individual stability properties of the discretized HJB and KFP equations with the bound on the $L^2$-norms of the density approximations from Lemma~\ref{lem:L2_strong_mono_bound} above to obtain a bound that relates to the overall stability of the discretized MFG system.}

{
\begin{lemma}\label{lem:apriori_mixed_norm}
Assume~\ref{ass:bounded} and \ref{ass:dmp}. For each $k\in \N$ and each $\overline{u}_k\in V_k$, define the nonnegative quantities
\begin{align}
\epsilon_k(\overline{u}_k) &\coloneqq \norm{u_{\lamk} - \overline{u}_k }_{H^1(\Omega)}+ \norm{m_{\lamk} - \mkp }_\Omega + \norm{h_{\Tk}\nabla u_{\lamk}}_\Omega + \lambda_k^{\frac{1}{2}},\label{eq:epsilon_k_def}
\\
\rho_k(\overline{u}_k) &\coloneqq \lambda_k^{-1}\norm{\RkH(\overline{u}_k,\mkp)}_{V_k^*}+\norm{\RkF(\overline{u}_k,\mkp)}_{V_k^*}.\label{eq:rho_k_def}
\end{align}
Then, there exists a $k_*\in \N$ such that
\begin{equation}\label{eq:apriori_mixed_norm_1}
\norm{\overline{u}_k-u_{k,\lamk}}_{H^1(\Omega)}+ \norm{\mkp-m_{k,\lamk}}_\Omega+\lambda_k \norm{\mkp-m_{k,\lamk}}_{H^1(\Omega)} \lesssim \epsilon_k(\overline{u}_k)+\rho_k(\overline{u}_k),
\end{equation}
for all $\overline{u}_k\in V_k$ and all $k\geq k_*$.
\end{lemma}

\begin{proof}
For each $k\in \N$, let $\overline{u}_k \in V_k$ be fixed but arbitrary.
We start by considering the error quantity $\mu_k(\overline{u}_k)$ defined by
\begin{equation}\label{eq:mu_k_def}
\mu_k(\overline{u}_k) \coloneqq \norm{\overline{u}_k-u_{k,\lamk}}_{H^1(\Omega)} + \lambda_k \norm{\mkp - m_{k,\lamk}}_{H^1(\Omega)} \quad \forall k \in \N.
\end{equation}
First, we apply the stability bound of the discrete KFP equation from Lemma~\ref{lem:regularized_KFP_stability} to the last term in $\mu_k(\overline{u}_k)$ to obtain
\begin{equation}
\mu_k(\overline{u}_k) \lesssim \norm{\overline{u}_k-u_{k,\lamk}}_{H^1(\Omega)}+\lambda_k \norm{\RkF(u_{k,\lambda_k},\mkp)}_{V_k^*} .
\end{equation}
Next, we apply the bound of Lemma~\ref{lem:KFP_res_consistency} for the dual norm $\norm{\RkF(u_{k,\lambda_k},\mkp)}_{V_k^*}$ to find that, for all $k$,
\begin{equation}
\begin{split}
\mu_k(\overline{u}_k) &\lesssim \norm{\overline{u}_k-u_{k,\lamk}}_{H^1(\Omega)} + \lambda_k \left( \lambda_k^{-1}  \norm{u_{\lamk} - u_{k,\lamk}}_{H^1(\Omega)} + \norm{m_{\lamk} - \mkp}_\Omega  \right)  
\\ & \lesssim \norm{\overline{u}_k- u_{k,\lamk}}_{H^1(\Omega)}+\norm{u_{\lamk} - \overline{u}_k}_{H^1(\Omega)} + \lambda_k \norm{m_{\lamk} - \mkp}_\Omega
\\ & \lesssim  \norm{\overline{u}_k- u_{k,\lamk}}_{H^1(\Omega)} + \epsilon_k(\overline{u}_k),
\end{split}
\end{equation}
where we have used the triangle inequality on $\norm{u_{\lamk} - u_{k,\lamk}}_{H^1(\Omega)}$ in order to pass to the second line above and we used the fact that $\{\lambda_k\}_{k\in\mathbb{N}}\subset (0,1]$ in passing to the third line.
Next, we use the stability bound for the discrete HJB equation of Lemma~\ref{lem:reg_HJB_discrete_stability} to bound $\norm{\overline{u}_k-u_{k,\lamk}}_{H^1(\Omega)}$ and thus find that, for all $k$ sufficiently large,
\begin{multline}
\mu_k(\overline{u}_k) 
  \lesssim  \norm{\RkH(\overline{u}_{k},\mkp)}_{V_k^*}+\lamk^{\frac{1}{2}} + \norm{\mkp-m_{k,\lamk}}_\Omega + \epsilon_k(\overline{u}_k),
\\
 \leq \norm{\RkH(\overline{u}_{k},\mkp)}_{V_k^*} +  \norm{\mkp-m_{k,\lamk}}_\Omega +  \epsilon_k(\overline{u}_k) ,
\end{multline}
where we recall that $\epsilon_k(\overline{u}_k)$ is defined in~\eqref{eq:epsilon_k_def} above. Then, we use the continuity bound~\eqref{eq:discrete_reg_HJB_consistency} which implies that $\norm{\RkH(\overline{u}_{k},\mkp)}_{V_k^*} \leq \epsilon_k(\overline{u}_k)$ to get
\begin{equation}\label{eq:apriori_mixed_norm_2}
\mu_k(\overline{u}_k) \lesssim \epsilon_k(\overline{u}_k)+ \norm{\mkp-m_{k,\lamk} }_\Omega,
\end{equation}
for all $k$ sufficiently large.
Then, we use Lemma~\ref{lem:L2_strong_mono_bound} which implies that
\begin{equation}\label{eq:apriori_mixed_norm_6}
\norm{\mkp-m_{k,\lamk}}_\Omega^2
\lesssim \norm{\RkH }_{V_k^*}\norm{\mkp-m_{k,\lamk}}_{H^1(\Omega)} + \norm{\RkF}_{V_k^*}\norm{\overline{u}_k-u_{k,\lamk}}_{H^1(\Omega)},
\end{equation}
where we have abbreviated the notation by omitting the arguments of the residual operators, i.e.\ $\RkH=\RkH(\overline{u}_k,\mkp)$ and $\RkF=\RkF(\overline{u}_k,\mkp)$ in \eqref{eq:apriori_mixed_norm_6} above.
Then, observe that $\norm{\RkH }_{V_k^*}\norm{\mkp-m_{k,\lamk}}_{H^1(\Omega)} \leq \rho_k(\overline{u}_k) \mu_k(\overline{u}_k) $ and also that $\norm{\RkF}_{V_k^*}\norm{\overline{u}_k-u_{k,\lamk}}_{H^1(\Omega)}\leq \rho_k(\overline{u}_k) \mu_k(\overline{u}_k)$, where it is recalled that $\rho_k(\overline{u}_k)$ is defined in~\eqref{eq:rho_k_def} above. 
Therefore, after taking square roots, we find that
\begin{equation}\label{eq:apriori_mixed_norm_7}
\norm{\mkp-m_{k,\lamk}}_\Omega \lesssim \rho_k(\overline{u}_k)^{\frac{1}{2}} \mu_k(\overline{u}_k)^{\frac{1}{2}}.
\end{equation}
Thus, combining~\eqref{eq:apriori_mixed_norm_2} and~\eqref{eq:apriori_mixed_norm_7}, we find that
\begin{equation}\label{eq:apriori_mixed_norm_3}
\mu_k(\overline{u}_k) \lesssim \epsilon_k(\overline{u}_k) + \rho_k(\overline{u}_k)^{\frac{1}{2}} \mu_k(\overline{u}_k)^{\frac{1}{2}}.
\end{equation}
We then apply Young's inequality to find that
\begin{equation}\label{eq:apriori_mixed_norm_4}
\mu_k(\overline{u}_k) \lesssim \epsilon_k(\overline{u}_k)+\rho_k(\overline{u}_k).
\end{equation}
Therefore, we see from~\eqref{eq:apriori_mixed_norm_7} and~\eqref{eq:apriori_mixed_norm_3} that
\begin{equation}\label{eq:apriori_mixed_norm_5}
\norm{\mkp-m_{k,\lamk}}_\Omega \lesssim \rho_k(\overline{u}_k)^{\frac{1}{2}}\left(\epsilon_k(\overline{u}_k)+\rho_k(\overline{u}_k)\right)^{\frac{1}{2}} \leq \epsilon_k(\overline{u}_k) + \rho_k(\overline{u}_k).
\end{equation}
Hence, we conclude~\eqref{eq:apriori_mixed_norm_1} upon adding~\eqref{eq:apriori_mixed_norm_4} to \eqref{eq:apriori_mixed_norm_5}.
\end{proof}
}

{
\paragraph{Proof of Lemma~\ref{lem:discrete_reg_l2_bound}}
Recall that the quantities $\epsilon_k(\overline{u}_k)$ and $\rho_k(\overline{u}_k)$ are defined in~\eqref{eq:epsilon_k_def} and~\eqref{eq:rho_k_def} above.
For arbitrary $\overline{u}_k\in V_k$, we start by applying the triangle inequality
\begin{multline}
\norm{u_{\lamk}-u_{\klk}}_{H^1(\Omega)} + \norm{m_{\lamk}-m_{\klk}}_\Omega 
 \\ \leq \norm{u_{\lamk}-\overline{u}_k}_{H^1(\Omega)} + \norm{m_{\lamk}-\mkp}_\Omega 
+ \norm{\overline{u}_k - u_{\klk}}_{H^1(\Omega)} + \norm{\mkp-m_{\klk}}_\Omega 
\\ \leq  \epsilon(\overline{u}_k) +  \norm{\overline{u}_k - u_{\klk}}_{H^1(\Omega)} + \norm{\mkp-m_{\klk}}_\Omega.
\end{multline}
It then follows from Lemma~\ref{lem:apriori_mixed_norm} that, for all $k$ sufficiently large,
\begin{equation}\label{eq:discrete_reg_l2_bound_1}
\norm{u_{\lamk}-u_{\klk}}_{H^1(\Omega)} + \norm{m_{\lamk}-m_{\klk}}_\Omega \lesssim \epsilon_k(\overline{u}_k) + \rho_k(\overline{u}_k) \quad \forall\,\overline{u}_k\in V_k.
\end{equation}
We now choose $\overline{u}_k = I^{\mathrm{SZ}}_k u_{\lamk} \in V_k$ to be the Scott--Zhang quasi-interpolation of $u_{\lamk}$~\cite{scott1990finite}, which, combined with the uniform $H^{1+\gamma}$ regularity bound~\eqref{eq:regularized_uniform_gamma} for $u_{\lamk}$, implies that
\begin{equation}\label{eq:discrete_reg_l2_bound_2}
\norm{u_{\lamk} - I^{\mathrm{SZ}}_k u_{\lamk}}_{H^1(\Omega)} \lesssim h_k^\gamma \quad \forall k\in \N.
\end{equation}
It remains now only to bound $\epsilon_k(I^{\mathrm{SZ}}_k u_{\lamk})$ and $\rho_k(I^{\mathrm{SZ}}_k u_{\lamk})$ for each $k\in\N$.
Recall that Lemma~\ref{lem:mkstar_l2_approx} shows that $\norm{m_{\lamk}-\mkp}_\Omega\lesssim h_k^\gamma$. It then follows that
\begin{equation}\label{eq:discrete_reg_l2_bound_3}
\epsilon_k(I^{\mathrm{SZ}}_k u_{\lamk}) \lesssim h_k^\gamma + h_k + \lambda_k^{\frac{1}{2}} \quad \forall k\in \N.
\end{equation}
To bound $\rho_k(I^{\mathrm{SZ}}_k u_{\lamk})$, we use Lemmas~\ref{lem:discrete_reg_HJB_consistency} and~\ref{lem:KFP_res_consistency} to find that
\begin{multline}
\rho_k(I^{\mathrm{SZ}}_k u_{\lamk}) = \lambda_k^{-1}\norm{\RkH(I^{\mathrm{SZ}}_k u_{\lamk},\mkp)}_{V_k^*} + \norm{\RkF(I^{\mathrm{SZ}}_k u_{\lamk},\mkp)}_{V_k^*}
\\\lesssim  \lambda_k^{-1}\left(\norm{u_{\lamk}-I^{\mathrm{SZ}}_k u_{\lamk}}_{H^1(\Omega)} + \norm{m_{\lamk}-\mkp}_\Omega + \norm{h_{\Tk}\nabla u_{\lamk}}_\Omega \right)
\\ +\lambda_k^{-1}\norm{u_{\lamk}-I^{\mathrm{SZ}}_k u_{\lamk}}_{H^1(\Omega)} + \norm{m_{\lamk}-\mkp}_\Omega \quad \forall k\in\N.
\end{multline}
Bounding each term above, and using the fact that $0<\lamk\leq 1$ and $0<\gamma\leq 1$, we then obtain
\begin{equation}\label{eq:discrete_reg_l2_bound_4}
\rho_k(I^{\mathrm{SZ}}_k u_{\lamk})\lesssim \lambda_k^{-1} h_k^\gamma + \lambda_k^{-1} h_k+h_k^{\gamma}\lesssim \lambda_k^{-1} h_k^{\gamma} \quad \forall k\in \N.
\end{equation}
It then follows from~\eqref{eq:discrete_reg_l2_bound_1}, \eqref{eq:discrete_reg_l2_bound_3} and~\eqref{eq:discrete_reg_l2_bound_4} that, for all $k$ sufficiently large,
\begin{equation}
\norm{u_{\lamk}-u_{\klk}}_{H^1(\Omega)} + \norm{m_{\lamk}-m_{\klk}}_\Omega \lesssim \lambda_k^{-1} h_k^\gamma + \lambda_k^{\frac{1}{2}}.
\end{equation}
This completes the proof of~\eqref{eq:discrete_reg_l2_bound} and thus of the Lemma.
\hfill\proofbox
}

{
\appendix

\section{Proof of Lemma \ref{lem:reg_discrete_convergence}}\label{appendix:pf-of-lemma-hk-lamk-vanish-jointly}
This appendix is dedicated to the proof of Lemma \ref{lem:reg_discrete_convergence}.

\subsection{Analysis of discrete regularized HJB equation}

In this subsection we establish convergence of the discrete HJB equation with regularized Hamiltonian to the continuous unregularized HJB equation. The analysis will be based on the following result for the discrete regularized HJB equation. It is proved similarly to \cite[Lemma 6.3]{osborne2022analysis}, so we omit the proof.
\begin{lemma}[Well-posedness and stability of discrete regularized HJB equation]\label{lemma-discr-reg-HJB-eqn-exist-uniq-stab}
	{Assume the hypotheses~\ref{ass:bounded} and~\ref{ass:dmp}.} 
	Let $\{\lambda_k\}_{k\in \N} \subset (0,1]$ denote a sequence of real numbers. Then, for each $\overline{m}\in L^2(\Omega)$ and $k\in\mathbb{N}$, there exists a unique ${w}_{k,\lamk}\in V_{k}$ such that
	\begin{equation}\label{eq:discr-reg-HJB-eqn}
		(A_k\nabla w_{k,\lambda_k},\nabla v_k)_\Omega+(H_{\lambda_k}[\nabla w_{k,\lambda_k}],v_k)_\Omega  = (F[\overline{m}],v_k)_\Omega \qquad\forall v_k\in V_k.
	\end{equation} 
	Furthermore,
	\begin{equation}\label{eq:disc_hjb_reg_a_priori_bound}
		\lVert w_{k,\lambda_k} \rVert_{H^1(\Omega)}  \lesssim  \lVert \overline{m} \rVert_\Omega + \lVert f \rVert_{C(\overline{\Omega}\times \mathcal{A})}+1 .
	\end{equation}
\end{lemma}

We now show the convergence of the discrete regularized HJB equation to the continuous HJB equation in the joint limit as the mesh-size and regularization parameter vanish.
\begin{lemma}[Convergence of discrete regularized HJB equation]\label{lemma-conv-disc-HJB-to-cont-HJB-eqn}
	{Assume the hypotheses~\ref{ass:bounded} and~\ref{ass:dmp}.} 
	Let $\{\lambda_k\}_{k\in \N} \subset (0,1]$ denote a sequence of real numbers converging to zero. If $\{\overline{m}_k\}_{k\in\mathbb{N}}$ is a sequence such that $\overline{m}_k\to \overline{m}$ in $L^2(\Omega)$ as $k\to\infty$, then the corresponding sequence of solutions $\{w_{k,\lambda_k}\}_{k\in\mathbb{N}}$ of the discrete HJB equation \eqref{eq:discr-reg-HJB-eqn} converges strongly to $\overline{u}$ in $H_0^1(\Omega)$ as $k\to\infty$, where $\overline{u}\in H_0^1(\Omega)$ {is the unique solution of}
	\begin{equation}\label{eq:weak-cont-HJB-eqn}
		\nu(\nabla \overline{u},\nabla v)_{\Omega} + (H[\nabla\overline{u}],v)_{\Omega} = (F[\overline{m}],v)_{\Omega}\quad\forall v\in H_0^1(\Omega).
	\end{equation} 
\end{lemma}
\begin{proof}
	Since $\{\lambda_k\}_{k\in \N} \subset (0,1]$ and $\{\overline{m}_k\}_{k\in\mathbb{N}}\subset L^2(\Omega)$ converges strongly in $L^2(\Omega)$ by hypothesis, we deduce from Lemma \ref{lemma-discr-reg-HJB-eqn-exist-uniq-stab}, in particular \eqref{eq:disc_hjb_reg_a_priori_bound}, that the sequence  $\{w_{k,\lambda_k}\}_{k\in\mathbb{N}}$ is uniformly bounded in $H_0^1(\Omega)$. Therefore, the Rellich--Kondrachov compactness theorem implies that we may pass to a subsequence (without change of notation) such that
	\begin{equation}\label{w_k_lam_k-conv}
		w_{k,\lambda_k}\rightharpoonup \overline{u} \text{ in }H_0^1(\Omega)\quad\text{and}\quad w_{k,\lambda_k}\to \overline{u} \text{ in }L^2(\Omega)
	\end{equation}
	as $k\to\infty$, for some $\overline{u}\in H_0^1(\Omega)$. Since the regularized Hamiltonian $H_{\lambda_k}$ has the linear growth property \eqref{eq:reg_H_linear_growth} uniformly in the regularization parameter, we also have that $\{H_{\lambda_k}[\nabla w_{k,\lambda_k}]\}_{k\in\mathbb{N}}$ is uniformly bounded in $L^2(\Omega)$. We therefore pass to a further subsequence, without change of notation, to obtain that 
	\begin{equation}\label{Hw_k_lam_k-cov}
		H_{\lambda_k}[\nabla w_{k,\lambda_k}]\rightharpoonup g\quad\text{in } L^2(\Omega)
	\end{equation}
	as $k\to\infty$ for some $g\in L^2(\Omega)$.
	Fix $j\in\mathbb{N}$ and let $v_j\in V_j$ be given. The hypothesis that the sequence of meshes $\{\mathcal{T}_k\}_{k\in\mathbb{N}}$ is nested implies that $V_j\subset V_k$ for all $k\geq j$. We then obtain that
	\begin{equation}
		(A_k\nabla w_{k,\lambda_k},\nabla v_j)_\Omega+(H_{\lambda_k}[\nabla w_{k,\lambda_k}],v_j)_\Omega  = (F[\overline{m}_k],v_j)_\Omega \qquad\forall k\geq j.
	\end{equation}
	Recalling \ref{ass:bounded} and Lipschitz continuity of $F$ c.f.\ \eqref{eq:F_lipschitz}, we use the convergences \eqref{w_k_lam_k-conv} and \eqref{Hw_k_lam_k-cov}, together with the strong convergence of $\{\overline{m}_k\}_{k\in\mathbb{N}}$ to $\overline{m}$ in $L^2(\Omega)$, to pass to the limit above to get
	\begin{equation}
		\nu (\nabla \overline{u},\nabla v_j)_\Omega+(g,v_j)_\Omega  = (F[\overline{m}],v_j)_\Omega \qquad\forall j\in\mathbb{N}.
	\end{equation}
	Since the union $\bigcup_{k\in\mathbb{N}}V_k$ is dense in $H_0^1(\Omega)$ and $j\in\mathbb{N}$ was arbitrary, we deduce that $\overline{u}$ satisfies
	\begin{equation}\label{eq:weak-cont-HJB-eqn_1}
		\nu( \nabla \overline{u},\nabla v)_\Omega+(g,v)_\Omega  = (F[\overline{m}],v)_\Omega \qquad\forall v\in H_0^1(\Omega).
	\end{equation} This implies in particular that $\|\nabla \overline{u}\|_{\Omega}^2=\nu^{-1}\left((F[\overline{m}],\overline{u})_{\Omega} - (g,\overline{u})_{\Omega}\right)$.
	On the other hand, for each given $k\in\mathbb{N}$, we take $w_{k,\lambda_k}\in V_k$ as a test function in the discrete HJB equation satisfed by $w_{k,\lambda_k}$ to find that
	\begin{equation}
		\nu\|\nabla w_{k,\lambda_k}\|_{\Omega}^2+(D_k\nabla w_{k,\lambda_k},\nabla w_{k,\lambda_k})_{\Omega}+(H_{\lambda_k}[\nabla w_{k,\lambda_k}],w_{k,\lambda_k})_{\Omega} = (F[\overline{m}_k],w_{k,\lambda_k})_\Omega.
	\end{equation}
	The hypothesis \ref{ass:bounded} and uniform boundedness of the sequence $\{w_{k,\lambda_k}\}_{k\in\mathbb{N}}$ in $H_0^1(\Omega)$ imply that 
	\begin{equation}
		\lim_{k\to\infty}|(D_k\nabla w_{k,\lambda_k},\nabla w_{k,\lambda_k})_{\Omega}|\leq \lim_{k\to\infty} \CD\norm{h_{\Tk}}_{L^\infty(\Omega)}\lVert w_{k,\lambda_k} \rVert_{H^1(\Omega)}^2=0.
	\end{equation}
	Therefore, the convergences \eqref{w_k_lam_k-conv} and \eqref{Hw_k_lam_k-cov}, together with the strong convergence of $\{\overline{m}_k\}_{k\in\mathbb{N}}$ to $\overline{m}$ in $L^2(\Omega)$ and the Lipschitz continuity of $F$, yield that
	\begin{equation}
		\lim_{k\to\infty}\nu\|\nabla w_{k,\lambda_k}\|_{\Omega}^2=(F[\overline{m}],\overline{u})_{\Omega}-(g,\overline{u})_{\Omega} =  \nu\|\nabla\overline{u}\|_{\Omega}^2.
	\end{equation}
	We therefore conclude from the $L^2$-convergence in \eqref{w_k_lam_k-conv} that $\|w_{k,\lambda_k}\|_{H^1(\Omega)}\to\|\overline{u}\|_{H^1(\Omega)}$ as $k\to\infty$. It then follows that the weak convergence of $\{w_{k,\lambda_k}\}_{k\in\mathbb{N}}$ to $\overline{u}$ in $H_0^1(\Omega)$ is in fact strong convergence: $w_{k,\lambda_k}\to\overline{u}$ in $H_0^1(\Omega)$ as $k\to\infty$ along a subsequence. The triangle inequality and the approximation bound \eqref{eq:reg_H_approx} satisfied by the regularized Hamiltonian then allow us to deduce that
	\begin{equation}
		\begin{split}
			&\|H[\nabla \overline{u}] - H_{\lambda_k}[\nabla w_{k,\lambda_k}]\|_{\Omega}
			\leq \|H[\nabla \overline{u}] - H[\nabla w_{k,\lambda_k}]\|_{\Omega}+{\frac{L_H^2\text{meas}_d(\Omega)^{\frac{1}{2}}\lambda_k }{2}}. 
		\end{split}
	\end{equation}
	This bound, along with the strong convergence of $\{w_{k,\lambda_k}\}_{k\in\mathbb{N}}$ to $\overline{u}$ in $H_0^1(\Omega)$, the weak convergence \eqref{Hw_k_lam_k-cov} and the convergence $\lambda_k\to 0$ as $k\to\infty$, yield that $H_{\lambda_k}[\nabla w_{k,\lambda_k}]\to g=H[\nabla \overline{u}]$ in $L^2(\Omega)$ as $k\to\infty$ along a subsequence. 
	{We then conclude from~\eqref{eq:weak-cont-HJB-eqn_1} that $\overline{u}\in H_0^1(\Omega)$ is a solution of the HJB equation~\eqref{eq:weak-cont-HJB-eqn}.}
	 Since this equation admits at most one solution in $H_0^1(\Omega)$ by \cite[Lemma 4.6]{osborne2022analysis}, we conclude that $\overline{u}$ is the unique solution of~\eqref{eq:weak-cont-HJB-eqn}. 
	 This implies that the strong convergence of a subsequence of $\{w_{k,\lambda_k}\}_{k\in\mathbb{N}}$ to $\overline{u}$ in $H_0^1(\Omega)$ holds also for the entire sequence $\{w_{k,\lambda_k}\}_{k\in\mathbb{N}}$, thereby completing the proof.
\end{proof}

\subsection{Convergence of solutions of regularized discrete MFG system \eqref{eq:regularized_fem_system} }

The proof of Lemma~\ref{lem:reg_discrete_convergence} that we present below is based on a compactness argument, c.f.\ \cite[Theorem 5.4]{osborne2022analysis}.
We begin by establishing uniform bounds for the sequence $\{(u_{k,\lambda_k},m_{k,\lambda_k})\}_{k\in\mathbb{N}}$ generated by the regularized discrete MFG system \eqref{eq:regularized_fem_system} with a given sequence of regularization parameters $\{\lambda_k\}_{k\in\mathbb{N}}\subset (0,1]$ converging to zero.

\paragraph{Proof of Lemma~\ref{lem:reg_discrete_convergence}.}

Due to the hypothesis \ref{ass:dmp}, we immediately deduce from \eqref{eq:infsup_stab_linop_2} that the density approximations $\{m_{k,\lambda_k}\}_{k\in\mathbb{N}}$ satisfy the uniform norm bound \eqref{eq:reg_discrete_m_apriori_bounds}.
The stability bound \eqref{eq:disc_hjb_reg_a_priori_bound} given by Lemma~\ref{lemma-discr-reg-HJB-eqn-exist-uniq-stab} and the fact that $\{\lambda_k\}_{k\in\mathbb{N}}\subset (0,1]$ converges to zero then imply that~\eqref{eq:reg_discrete_u_apriori_bounds} {holds, which shows that the sequence $\{ u_{\klk}  \}_{k\in \N}$ is uniformly bounded in $H^1$.}
This proves {the uniform bounds on the norms of the approximations, as stated in} Lemma~\ref{lem:reg_discrete_convergence}. 
The uniform bounds \eqref{eq:reg_discrete_m_apriori_bounds} and \eqref{eq:reg_discrete_u_apriori_bounds} then imply we may pass to a subsequence, without change of notation, such that
\begin{subequations}
	\begin{alignat}{2} 
		\label{eq:uklamk-conv}
		u_{k,\lambda_k}\rightharpoonup \overline{u} &\quad\text{in}\quad H_0^1(\Omega) &\quad&u_{k,\lambda_k}\to\overline{u}\quad\text{in}\quad L^2(\Omega),
		\\ 
		\label{eq:mklamk-conv}
		m_{k,\lambda_k}\rightharpoonup \overline{m} &\quad\text{in}\quad H_0^1(\Omega) &\quad&m_{k,\lambda_k}\to\overline{m}\quad\text{in}\quad L^q(\Omega),
	\end{alignat}
\end{subequations}
as $k\to\infty$, for some $\overline{u},\overline{m}\in H_0^1(\Omega)$, for any $q\in [1,2^*)$, where $2^*\coloneqq \infty$ if $d=2$ and $2^*\coloneqq \frac{2d}{d-2}$ if $d\geq 3$, and for any $q\in[1,\infty]$ if $d=1$. Then, the convergences \eqref{eq:uklamk-conv} and \eqref{eq:mklamk-conv}, together with Lemma \ref{lemma-conv-disc-HJB-to-cont-HJB-eqn} on convergence of the discrete HJB equation, imply that $u_{k,\lambda_k}\to \overline{u}$ in $H_0^1(\Omega)$ as $k\to\infty$ along a subsequence, where $\overline{u}\in H_0^1(\Omega)$ is the unique solution of 
\begin{equation}\label{eq:weak-cont-HJB-eqn-final-pf}
	\nu(\nabla \overline{u},\nabla v)_{\Omega} + (H[\nabla\overline{u}],v)_{\Omega} = (F[\overline{m}],v)_{\Omega}\quad\forall v\in H_0^1(\Omega).
\end{equation}
The uniform $L^{\infty}$-bound \eqref{eq:reg_H_deriv_linf_bound} implies that we may pass to a further subsequence (without change of notation) such that
\begin{equation}\label{eq:bklamk-conv}
	\frac{\partial H_{\lambda_k}}{\partial p}[\nabla u_{k,\lambda_k}]\rightharpoonup^* \tilde{b}_* \quad\text{in}\quad L^{\infty}(\Omega;\mathbb{R}^d),
\end{equation}
for some vector field $\tilde{b}_*\in L^{\infty}(\Omega;\mathbb{R}^d)$. Since the Moreau--Yosida regularization of the Hamiltonian satisfies \cite[Hypothesis (H5)]{osborne2024regularization} and we have that  $u_{k,\lambda_k}\to \overline{u}$ in $H_0^1(\Omega)$ as $k\to\infty$ along a subsequence, we invoke \cite[Lemma 5.3]{osborne2024regularization} to deduce that $\tilde{b}^*$ satisfies
\begin{equation}\label{eq:tildeb_*-incl-pf}
	\tilde{b}_*\in D_pH[\overline{u}].
\end{equation}
Consequently, we have that \eqref{eq:mklamk-conv} and \eqref{eq:bklamk-conv} imply  $m_{k,\lambda_k}\frac{\partial H_{\lambda_k}}{\partial p}[\nabla u_{k,\lambda_k}]\rightharpoonup \overline{m}\tilde{b}_* $ weakly in $L^{2}(\Omega;\mathbb{R}^d)$ as $k\to\infty$. Furthermore, the stability bound \eqref{eq:reg_discrete_m_apriori_bounds} and the hypothesis \ref{ass:bounded} give, for any $v\in H_0^1(\Omega)$,
\begin{equation}
	\lim_{k\to\infty}|(D_k\nabla m_{k,\lambda_k},\nabla v)_{\Omega}|\leq \lim_{k\to\infty} \CD\norm{h_{\Tk}}_{L^\infty(\Omega)}\lVert m_{k,\lambda_k} \rVert_{H^1(\Omega)}\lVert v \rVert_{H^1(\Omega)}=0.
\end{equation}
Fix $j\in\mathbb{N}$ and let $w_j\in V_j$ be fixed but arbitrary.
The hypothesis that the sequence of meshes $\{\mathcal{T}_k\}_{k\in\mathbb{N}}$ is nested implies that $V_j\subset V_k$ for each $k\geq j$, so by passing to the limit in the discrete KFP equation
\begin{equation}
	(A_k\nabla m_{k,\lambda_k},\nabla w_j)_\Omega+\left(m_{k,\lambda_k}\frac{\partial H_{\lambda_k}}{\partial p}[\nabla u_{k,\lambda_k}],\nabla w_j\right)_\Omega  =\langle G,w_j\rangle_{H^{-1}\times H^1_0}
\end{equation}
where we send $k\to\infty$, we conclude that $\overline{m}$ satisfies
\begin{equation}\label{eq:m-pre-lim-eqn}
	\nu(\nabla \overline{m},\nabla w_j)_\Omega+\left(\overline{m}\tilde{b}_*,\nabla w_j\right)_\Omega  =\langle G,w_j\rangle_{H^{-1}\times H^1_0}.
\end{equation}
As the union $\bigcup_{k\in\mathbb{N}}V_k$ is dense in $H_0^1(\Omega)$ and the equation \eqref{eq:m-pre-lim-eqn} holds for all $w_j\in V_j$, $j\in\mathbb{N}$, we deduce {that $\overline{m} \in H^1_0(\Omega)$ satisfies}
\begin{equation}\label{eq:KFP-pf-klamk-conv}
	\nu(\nabla \overline{m},\nabla w)_\Omega+\left(\overline{m}\tilde{b}_*,\nabla w\right)_\Omega  =\langle G,w \rangle_{H^{-1}\times H^1_0}\quad\forall w\in H_0^1(\Omega).
\end{equation}
In light of \eqref{eq:weak-cont-HJB-eqn-final-pf}, \eqref{eq:tildeb_*-incl-pf} and \eqref{eq:KFP-pf-klamk-conv}, we see then that $(\overline{u},\overline{m})\in H_0^1(\Omega)\times H_0^1(\Omega)$ is a weak solution of the MFG PDI \eqref{eq:PDI_weakform}. But uniqueness of the solution of the weak formulation \eqref{eq:PDI_weakform} (see \cite[Theorem 3.4]{osborne2022analysis}) then implies that $\overline{u}=u$ and $\overline{m}=m$ in $H_0^1(\Omega)$. 
In summary, we have shown that the sequence of approximations $\{(u_{k,\lambda_k},m_{k,\lambda_k})\}_{k\in\mathbb{N}}$ converges to the weak solution of the MFG PDI \eqref{eq:PDI_weakform} along a subsequence in the sense of \eqref{eq:mklamk-conv} with $u_{k,\lambda_k}\to {u}$ in $H_0^1(\Omega)$. Uniqueness of the weak solution then implies that the entire sequence $\{(u_{k,\lambda_k}, m_{k,\lambda_k})\}_{k\in\mathbb{N}}$ converges to $(u,m)$ in $H^1\times L^q$ in the sense of \eqref{eq:reg_discrete_convergence}. This completes the proof of Lemma \ref{lem:reg_discrete_convergence}.
\hfill\proofbox
}

\section*{Acknowledgments}
The work of the first author was supported by The Royal Society Career Development Fellowship.
The work of the second author was supported by the Engineering and Physical Sciences Research Council [grant number EP/Y008758/1].

\end{document}